\documentclass[11pt]{article}
\usepackage[utf8]{inputenc}
\usepackage{amsmath, amsfonts, amssymb, amsthm}
\usepackage{geometry}
\usepackage{bm}
\usepackage{booktabs}
\usepackage{longtable}
\usepackage{hyperref}
\usepackage{graphicx}

\newtheorem{theorem}{Theorem}
\newtheorem{lemma}[theorem]{Lemma}
\newtheorem{proposition}[theorem]{Proposition}
\newtheorem{definition}{Definition}
\newtheorem{remark}{Remark}

\newtheorem{corollary}{Corollary}
\newtheorem{problem}{Problem}
\newtheorem{example}{Example}

\begin{document}

\title{Information Design under Uncertain Utilities: Probabilistic and CVaR Approaches}
\author{Furkan Sezer\thanks{Texas A\&M University.}}
\date{\today}

\maketitle
\begin{abstract}
This paper studies information design when the designer lacks precise knowledge of agents' payoff coefficients. The Calibrated Bayes Correlated Equilibrium (Cal-BCE) is introduced as a solution concept that augments the Bayes correlated equilibrium with a corrector policy preserving incentive compatibility under the designer's structural uncertainty, adapting its revelation principle to this setting. The design problem is nonconvex in general, but under a linear-quadratic-Gaussian structure it admits convex second-order cone and semidefinite reformulations under two-sided probabilistic and conditional value-at-risk (CVaR) constraints, with feasibility guaranteed by a Hadamard invertibility condition. A joint decentralization theorem shows that both designs cap cross-agent action covariances, the CVaR design more tightly at a common tolerance; but because the formulations operate at design-specific feasibility thresholds, the realized ordering is calibration-dependent. Experiments on fifteen sector ETFs confirm the trade-off: the probabilistic design attains higher mean welfare and the CVaR design better tail protection, with neither dominating outright.
\end{abstract}


\section{Introduction}

\emph{Information design} studies how a designer can influence equilibrium outcomes in \emph{incomplete-information games} by controlling the distribution of signals about underlying payoff states.  
The designer commits to a rule that maps payoff-relevant states into signals observed by agents before the realization of the states, thereby shaping their posterior beliefs, best responses, and the resulting equilibrium.  
As formalized in \cite{Bergemann2019}, the information-design problem can be formulated as an optimization problem in which the designer selects a feasible signal distribution consistent with a common prior to optimize an objective function subject to equilibrium constraints among agents.  
This framework complements classical mechanism design by focusing on the choice of information rather than incentive schemes, and provides a foundation for analyzing how information policies affect welfare, efficiency, and systemic stability across strategic environments.

Information design has been extensively studied in market and auction contexts, where disclosure of value information changes bidding incentives, efficiency, and revenue outcomes. In second-price auctions, the revenue-maximizing seller fully reveals low valuations but pools high ones to balance efficiency and information rents \cite{Bergemann2022AERInsights}. Related work provides closed-form characterizations of buyer-optimal and seller-worst information structures in Myerson auctions \cite{ChenYang2023JET}, demonstrates that randomization over information structures can enable full surplus extraction \cite{Krahmer2020JET}, and analyzes how correlated information affects equilibrium bidding and revenue in first-price auctions \cite{Bergemann2017Ecta}.  
These results establish that optimal information disclosure balances allocative efficiency with incentive rents, a principle that recurs across broader economic and financial environments.

Beyond markets and auctions, information design methods have been applied to a range of settings characterized by strategic interdependence and uncertainty. On the technical side, constrained \cite{doval2024Constrained}, multiple-designer \cite{Koessler2022Interactive} and privately-informed-agent \cite{candogan2023private} extensions of information design are studied. In congestion games, appropriately garbled signals can improve routing efficiency under stochastic demand \cite{DasKamenicaMirka2017,Griesbach2023}. In addition, the disclosure of information under monetary incentives does not worsen congestion \cite{ferguson2024monetary}.
In public health and epidemic control, selective transparency about infection or vaccination states can guide behavioral responses and reduce aggregate contagion \cite{alizamir2020healthcare,shah2022optimal,pathak2022epidemic}.  
In operations and supply-chain management, strategically releasing demand or inventory information can coordinate decentralized decisions and mitigate inefficiencies \cite{Beranek2024EJOR}. Similarly, \cite{candogan2025value} studies information design in two-tier supply chains. These diverse applications share a common structure: a designer selects an information policy to influence equilibrium outcomes in a system of interacting agents.

In financial and macroeconomic systems, the information itself functions as a policy instrument. In macroprudential stress testing, disclosure rules determine how much scenario and performance information is revealed to markets and regulators, shaping capitalization outcomes, stigma effects, and contagion \cite{OrlovZryumovSkrzypacz2023,AlonsoZachariadis2021,Inostroza2020}.  
Optimal stress tests may deliberately fail both weak and some strong institutions to limit stigma, while sequential designs outperform static ones by staging recapitalization and public assessments \cite{OrlovZryumovSkrzypacz2023}.  
When large investors possess private information, the precision of public tests can induce either \emph{crowding-in} or \emph{crowding-out} of private capital \cite{AlonsoZachariadis2021}. With multiple audiences—creditors, equity investors, and supervisors—stress-test communication may be opaque for strong banks and transparent for weak ones, aligning disclosure with prudential goals \cite{Inostroza2020}. Analogous principles appear in monetary policy and central banking, where the communication of rate intentions or inflation forecasts can be formulated as an information design problem balancing credibility and stabilization \cite{Tamura2023}. At the systemic level, disclosure interacts with network fragility and contagion: optimal communication becomes conservative when spillovers dominate risk-sharing \cite{FanTang2023}, and network density can shift from stabilizing to destabilizing as shocks grow, linking information to systemic risk \cite{AcemogluOzdaglarTahbazSalehi2015}.

Despite the expanding scope of information design theory, most existing formulations share a critical assumption: the designer has precise knowledge of agents’ payoff functions and the mapping between underlying states and utilities.  
This assumption rarely holds in complex financial or economic systems, where payoff relationships between sectors, assets, or institutions are themselves uncertain and estimated from noisy data. A regulator designing disclosure rules, an exchange publishing volatility indicators, or a policymaker releasing macroeconomic signals must all make information decisions under uncertainty about the true structure of interdependencies. Such uncertainty is not purely epistemic—about beliefs or priors—but \emph{structural}, reflecting incomplete knowledge of the system’s payoff coefficients.

Addressing this challenge requires an equilibrium notion that remains well-defined when the designer’s model of the game is itself uncertain.  
To this end, this paper introduces the \emph{Calibrated Bayes Correlated Equilibrium (Cal-BCE)}, a new solution concept that extends the Bayes correlated equilibrium to settings where payoff coefficients are random variables drawn from a known distribution.  
Cal-BCE incorporates an explicit correction term that reconciles the information designer’s uncertainty with agents’ equilibrium strategies, ensuring consistency between the designer’s optimization and the agents’ best responses.  

In general, the resulting information-design problem is nonconvex, but under the linear–quadratic–Gaussian (LQG) structure, Cal-BCE admits a convex representation. 
Specifically, when utilities are quadratic and payoff states and information structures are Gaussian, the equilibrium feasibility set becomes semidefinite and SOCP representable, allowing tractable formulations under Conditional Value-at-Risk (CVaR) and probabilistic constraints, respectively.
This yields a convex optimization framework that interpolates between expected-value efficiency and tail-risk robustness, while remaining analytically transparent and computationally tractable.  
The LQG formulation thus provides a precise setting in which Cal-BCE can be solved via semidefinite or second-order cone programs, linking information design with risk-sensitive optimization. Beyond tractability, I prove three structural results: an adaptation of the revelation principle for Bayes correlated equilibrium \cite{Bergemann2013,Bergemann2019} to environments with payoff uncertainty (Proposition~\ref{prop:revelation}); a feasibility guarantee under a checkable Hadamard invertibility condition (Proposition~\ref{prop:feasibility}); and a joint decentralization theorem showing that both formulations cap cross-agent action covariances and suppress strategic interdependence, with the CVaR formulation providing tighter suppression at a common tolerance while the realized ordering between the two is governed by their feasibility thresholds (Theorem~\ref{thm:decentralization}). The following subsection situates this contribution within the growing literature on information design under uncertainty and highlights how Cal-BCE complements learning-based and worst-case robust approaches.

\subsection{Literature review on information design under uncertainty.}
Classical information design assumes that the designer possesses full knowledge of agents' utilities and beliefs, an assumption that recent studies have begun to relax in different ways.  
Li and Lin \cite{li2025informationdesignunknownprior} study \emph{information design with unknown receiver priors} and develop sequential persuasion algorithms that achieve logarithmic regret as the designer learns the receiver’s belief through repeated interactions.  
Zu et al. \cite{zu2021learning} similarly investigate a repeated persuasion setting in which both sender and receiver are ignorant of the prior, proposing a robustness criterion that ensures persuasiveness across candidate priors and establishing \(O(\sqrt{T \log T})\) regret bounds.  
Kosterina \cite{svetlana2022unknownbelief} explores persuasion when the designer does not know the receivers’ beliefs, analyzing how incomplete knowledge of belief formation shapes optimal information structures.  
Babichenko et al. \cite{BABICHENKO2022226} study \emph{regret-minimizing Bayesian persuasion} when the sender is ignorant of the receiver’s utility function and quantify how much regret the sender can guarantee despite this ignorance, highlighting a form of robustness grounded in regret minimization.  
Dworczak and Pavan \cite{dworczak2022preparing} take a complementary approach by formulating a \emph{robust Bayesian persuasion} model in which the sender prepares for the worst-case conjecture about the receiver’s type or information environment and then optimizes within the set of such robustly optimal policies.  
Collectively, these works primarily address \emph{epistemic uncertainty} in agents’ beliefs or preferences.

In contrast, the present framework focuses on \emph{structural uncertainty} in the payoff environment of a multi-agent game.  
The proposed \emph{Calibrated Bayes Correlated Equilibrium (Cal-BCE)} introduces an explicit correction mechanism that maintains equilibrium consistency when the designer does not know the true payoff coefficients but only their distribution.  
Rather than learning or minimizing worst-case regret, Cal-BCE yields a static convex characterization of equilibrium feasibility under payoff uncertainty, implemented through probabilistic and CVaR-based constraints.  
This perspective complements robust optimization approaches \cite{sezer_robust2023}, which focus on worst-case realizations of the payoff structure, by providing a distributionally risk-sensitive alternative that interpolates between efficiency and robustness.  
Together, these advances extend the scope of information design from fully specified Bayesian environments toward settings where the designer must operate under incomplete or uncertain knowledge of the underlying game structure.

\noindent \textbf{Contributions:}
This paper studies the problem of information design when the designer does not have full knowledge of agents' utilities. The contributions are as follows.

\noindent\emph{(i) Cal-BCE and a revelation argument under payoff uncertainty.} I introduce the Calibrated Bayes Correlated Equilibrium (Cal-BCE), a solution concept that extends the Bayes correlated equilibrium of \cite{Bergemann2013,Bergemann2019} to settings where the payoff coefficients $\theta=\mathrm{diag}(H)$ are unknown to the designer but known to the agents. Adapting the revelation principle for Bayes correlated equilibrium to this setting---by conditioning on $\theta$ and taking expectations over the designer's uncertainty---I show (Proposition~\ref{prop:revelation}) that optimizing over information structures and corrector policies reduces without loss to optimizing over action distributions subject to Cal-BCE obedience constraints. The contribution here is the corrector-policy construction and the resulting calibrated constraints, not the revelation argument itself, which follows the standard logic.

\noindent\emph{(ii) Tractable convex reformulations.} Under the LQG structure, I derive convex SDP/SOCP formulations for both the CVaR (Theorem~\ref{theorem_cvar}) and probabilistic (Theorem~\ref{theorem_chance}) Cal-BCE programs. The equilibrium constraints are enforced as two-sided bands over the payoff uncertainty distribution, with the chance program using Bonferroni-corrected per-side confidence $(1+\beta)/2$ to guarantee coverage at the prescribed level $\beta$.

\noindent\emph{(iii) Feasibility guarantee.} I provide an explicit sufficient condition (Proposition~\ref{prop:feasibility}, Theorem~\ref{thm:general_existence}): the feasible region is nonempty whenever $\bar{H}\circ\Sigma$ is invertible and the tolerance satisfies $\varepsilon\ge\varepsilon^\star=\kappa_\beta\max_i s_i$, with $\kappa_\beta$ the per-formulation tail multiplier. The certificate is a diagonal linear policy with gains $d=(\bar{H}\circ\Sigma)^{-1}\mathrm{dg}(\Sigma)$ in closed form.

\noindent\emph{(iv) Joint decentralization theorem.} I prove (Theorem~\ref{thm:decentralization}) that \emph{both} the probabilistic and CVaR designs suppress cross-agent action covariances, with each off-diagonal entry capped by $|X_{ij}|\le\varepsilon/(\kappa_\beta\sigma_{ij})$. At a common tolerance the CVaR multiplier $\kappa_\beta^{\mathrm{CVaR}}\ge\kappa_\beta^{\mathrm{chance}}\ge\Phi^{-1}(\beta)$ ($\beta\in[0.9,0.99]$) yields a tighter cap; but since the two formulations operate at design-specific feasibility thresholds $\varepsilon^\star$, the \emph{realized} decentralization ordering is governed by the threshold ratio $\varepsilon^\star_{\mathrm{CVaR}}/\varepsilon^\star_{\mathrm{chance}}$ and is therefore calibration-dependent---reconciling the theory with the empirically observed ordering.

\section{Information Design Problem under Parameter Uncertainty}
An incomplete information game involves a set of $n$ players belonging to the set $\mathcal{N}:=\{1,\dots,n\}$, each of which selects actions $a_i \in \mathcal{A}_i$ to maximize the expectation of its individual payoff function $u_i(a,\gamma)$ where $a \equiv (a_{i})_{i \in \mathcal{N}} \in \mathcal{A} $ and $\gamma \equiv (\gamma_{i})_{i \in \mathcal{N}} \in \Gamma$ correspond to an action profile and an unknown payoff state vector, respectively (See Appendix \ref{app_notation} for a key notation table). I refer to the payoff state that directly influences agent $i$'s action selection by agent $i$'s payoff state $\gamma_{i}$. The agent $i$ forms expectations about its payoff based on its signal/type $\omega_i \in \Omega_i$ about the state $\gamma_{i}$. I denote by $\theta \in \Theta$ the payoff coefficients that are unknown to the designer but known to the agents. In the LQG specialization (Section~\ref{sec:lqg}) the uncertain primitive is the payoff matrix $H\sim\xi$, and $\theta=\mathrm{diag}(H)\in\mathbb{R}^n$ is the vector of agent-specific own-coefficients ($\theta_i=H_{ii}$) against which the per-agent corrector is calibrated; the designer's uncertainty about the remaining (cross-agent) coefficients of $H$ is carried by the distribution $\xi$ and handled through the robust action-channel constraints developed in Section~\ref{sec:lqg}. I represent the incomplete information game by the tuple $G:= \{\mathcal{N}, \mathcal{A}, \Gamma,\Theta,\{u_i\}_{i\in \mathcal{N} }, \{\omega_i\}_{i\in \mathcal{N}}\}$.  The strategy of player $i$ maps each possible value of the private signal  $\omega_{i} \in \Omega_{i}$ to an action $s_{i}(\omega_{i}) \in \mathcal{A}_{i}$, i.e., $s_{i}: \Omega_{i} \rightarrow \mathcal{A}_{i}$. 

A strategy profile $s = (s_{i})_{i \in \mathcal{N}}$ is a Bayesian Nash equilibrium with information structure $\zeta$, if it satisfies the following inequality
\begin{equation}
E_{\zeta}[u_{i}(s_{i}(\omega_{i}), s_{-i},\gamma,\theta )|\omega_{i}] \geq  E_{\zeta}[u_{i}(a_{i}', s_{-i},\gamma, \theta )|\omega_{i}],\label{eq_class_bne} 
\end{equation}
for all $ a_{i}' \in \mathcal{A}_{i}, \omega_{i} \in \Omega_{i}, i \in \mathcal{N}$ where $s_{-i}=  (s_{j}(\omega_{j}))_{j\neq i}$ is the equilibrium strategy of all the players except player $i$, and $E_{\zeta}$ is the expectation operator with respect to the distribution $\zeta$ and the prior on the payoff state $\psi$. 

An information designer aims to maximize an objective function at the system level $f:\mathcal{A}\times \Gamma\to \mathbb{R}$, e.g. social welfare, that depends on the actions of the players and the state realization by deciding on an information structure $\zeta(\omega|\gamma)$ belonging to the feasible space of probability distributions in the signal space $\mathcal{Z}$ given a game with unknown payoffs.  The information structure determines the fidelity of the signals $\{\omega_i\}_{i\in \mathcal{N}}$ that will be revealed to the players given a realization of the payoff state $\gamma$. 

The information designer does not know that actual payoffs but knows that the payoffs $u:=\{u_i\}_{i\in \mathcal{N}}$ belong to a space of payoff functions $U_\theta$ which is parameterized by $\theta.$  I use $\mathcal{G}_\Theta$ to refer to the set of games with payoffs $u\in U_\theta$. Specifically,  $U_\theta$ refers to utility functions which can realize due to payoff coefficient distribution $\xi$ i.e $\theta \sim \xi$. Designer needs to ensure strategies given to agents form an equilibrium, therefore, it needs to consider corrector distribution which I pose to handle information design problem under uncertainty. Corrector distribution $\lambda(\alpha|\theta)$ accounts for deviation to equilibrium caused by unknown parameters $\theta$ where $\alpha $ is the corrector profile to strategies.

\subsection{Information Design via Action Distributions}
Two information design problems under uncertainty are posed below. Instead of working with the information structure $\zeta,$ the analysis will describe problems using action distributions that I formally define below. 

\begin{definition}[Action distribution]
An action distribution is the probability of observing an action profile $a\in \mathcal{A}$ when agents follow a strategy profile $s$ under $\zeta$, which can be computed as %
\begin{equation}\label{eq_phi}
\phi(a|\gamma) = \sum_{\omega: s(\omega)=a}\zeta(\omega|\gamma). 
\end{equation}
\end{definition}
According to the definition, the probability of observing the action profile $a$ is the sum of the conditional probabilities of all the signal profiles $\omega$ under $\zeta$ that induce the action profile $a$ given the strategy profile $s$.

 A calibrated Bayes correlated equilibrium (Cal-BCE) is an action distribution in which no rational agent would profit by unilaterally deviating from equilibrium actions under recommended actions.

 \begin{definition}[Calibrated Bayes correlated equilibrium]\label{def:calbce}
An action distribution $\phi$ under information structure $\zeta$ is a Cal-BCE if and only if it satisfies
\begin{equation}\label{eq:bce-cond}
E_{\lambda}[E_{\phi}[u_{i}((a_{i}-\alpha_{i},a_{-i}),\gamma,\theta)|a_{i}]] \geq E_{\lambda}[ E_{\phi}[u_{i}((a'_{i},a_{-i}),\gamma,\theta)|a_{i}]],  
\end{equation}
$\forall a_{i}, a'_{i} \in \mathcal{A}_{i},  \; i\in \textbf{N}$ where $E_{\phi}[\cdot | a_i]$ is the conditional expectation with respect to the action distribution $\phi$ and information structure $\zeta$ given action $a_i\in \mathcal{A}_i$. I denote the corrector parameter by $\alpha_{i}$ with $\theta$ the corrector distribution given in $\lambda(\alpha|\theta).$ 
\end{definition}

\paragraph{Signal structure under Cal-BCE.}
In the standard Bayes correlated equilibrium (BCE), each agent $i$ receives a private signal $\omega_i$ drawn from the information structure $\zeta(\omega|\gamma)$ and selects an action according to its equilibrium strategy $s_i(\omega_i)$.  
A well-known result \cite{Bergemann2013} shows that when the system parameters are fully known, the designer can equivalently recommend actions instead of sending signals—i.e. $\Omega_i = \mathcal{A}_i$—since the revelation principle ensures no loss of generality.  

However, under payoff uncertainty, this direct equivalence fails because the designer does not know the true utility coefficients $\theta$.  
To preserve equilibrium consistency, the Calibrated Bayes Correlated Equilibrium (Cal-BCE) introduces a correction distribution $\lambda(\alpha|\theta)$ that adjusts the recommended actions to offset potential misspecification.  
Specifically, the designer recommends
\begin{equation}
a_i = \omega_i + \alpha_i,
\end{equation}
where $\omega_i$ is drawn from the intended information policy and $\alpha_i$ is sampled from $\lambda(\alpha|\theta)$.  
Agents then form their best responses with respect to the \emph{effective signal}
\begin{equation}
\tilde{\omega}_i = a_i - \alpha_i,
\end{equation}
which represents the equilibrium-consistent recommendation accounting for the designer’s uncertainty.  
Equivalently, the induced information structure over effective signals can be expressed as
\begin{equation}
\tilde{\zeta}(\tilde{\omega}|\gamma,\theta)
= \int \delta(\tilde{\omega} - (\omega - \alpha))\,\zeta(\omega|\gamma)\,\lambda(\alpha|\theta)\,d\omega\,d\alpha,
\label{eq:tilde-zeta}
\end{equation}
demonstrating that Cal-BCE can be viewed as a standard BCE operating over calibrated signals.  
This perspective clarifies that while the designer recommends actions rather than raw signals, the correction term $\alpha_i$ effectively adjusts these recommendations to ensure equilibrium consistency under payoff uncertainty. In addition, a constructive existence result establishing a nontrivial Cal-BCE with explicit correction under parameter uncertainty is provided in Theorem \ref{thm:general_existence} (d) under  Appendix~\ref{app:Cal-BCE_existence}.

\paragraph{Commitment and information timing.}
The designer commits, \emph{before} the realization of $\gamma$ and \emph{before} learning the payoff matrix $H$, to two objects: (i) an information policy $\zeta$ (equivalently, the action covariance $X$), and (ii) a \emph{corrector policy} $\lambda(\alpha\mid\theta)$, i.e.\ a \emph{rule} mapping the realized uncertainty state $\theta=\mathrm{diag}(H)$ to a distribution over corrections $\alpha$. The commitment is to the policy $\lambda(\cdot\mid\cdot)$, not to any fixed correction; the rule is evaluated only once $\theta$ is realized. The interaction proceeds as follows:
\begin{enumerate}
\item The designer announces $\phi^\star$ (equivalently $\zeta^\star$) and the corrector policy $\lambda(\cdot\mid\cdot)$.
\item Nature draws $(\gamma,H)$ from the prior $\psi\otimes\xi$. The agents observe their own payoff coefficients, hence $\theta$; the designer does not.
\item Agent $i$ receives the recommendation $a_i=\omega_i+\alpha_i$, with $\omega_i$ drawn from $\zeta^\star$ and $\alpha_i$ from $\lambda(\cdot\mid\theta)$. Since $\alpha_i$ is $\theta$-measurable and $\lambda$ is common knowledge, agent $i$---who knows $\theta$---recovers the effective signal $\tilde\omega_i=a_i-\alpha_i$.
\item Agents best-respond to $\tilde\omega_i$; the played profile is $a-\alpha$.
\end{enumerate}
This resolves the apparent circularity of committing to a $\theta$-dependent object under designer uncertainty: the designer commits to the \emph{mapping} $\lambda$, whereas its \emph{evaluation} happens on the agent side, where $\theta$ is known. Accordingly, the obedience requirement in Definition~\ref{def:calbce} is an \emph{expectation}-level condition (over $\lambda$ and the conditional action distribution): the correction is chosen so that, in conditional expectation, playing the corrected recommendation $\tilde\omega_i=a_i-\alpha_i$ is a best response---even though the designer cannot compute the obedient recommendation directly without knowing $H$.

\paragraph{Intuition.}
Under payoff uncertainty, the designer commits to a information structure using a misspecified model, while agents evaluate recommendations under the true coefficients, breaking standard incentive compatibility. The corrector $\alpha_i$ resolves this: by recommending $a_i = \omega_i + \alpha_i$, it allows agents—who observe the true payoffs—to filter out $\alpha_i$ and recover a calibrated signal $\tilde{\omega}_i = a_i - \alpha_i$ aligned with their true best responses. 

Crucially, this expands implementability. While the designer cannot directly condition recommendations on the realized parameters $\theta$, the mapping $\lambda(\alpha \mid \theta)$ allows agents to perform this conditioning ex post using their private information. Cal-BCE thus replicates a $\theta$-contingent mechanism, extending the revelation principle of \cite{Bergemann2019} to environments with structural uncertainty (see Proposition~\ref{prop:revelation}).

\begin{proposition}[Revelation principle under payoff uncertainty]\label{prop:revelation}
Fix a prior $\psi$ over $\gamma$ and a distribution $\xi$ over payoff matrices $H$, and let $\theta=\mathrm{diag}(H)$ be the uncertainty state known to the agents. For any information structure $\zeta$ and corrector policy $\lambda(\alpha\mid\theta)$ whose induced effective signals $\tilde\omega=\omega-\alpha$ satisfy the obedience condition \eqref{eq:bce-cond}, there is an outcome-equivalent \emph{direct} mechanism in which the designer recommends actions $a_i=\omega_i+\alpha_i$ and each agent $i$, knowing $\theta_i$, finds it optimal in conditional expectation to play the corrected recommendation $\tilde\omega_i=a_i-\alpha_i$. Consequently, optimizing the designer's objective over pairs $(\zeta,\lambda)$ is without loss equivalent to optimizing over action distributions $\phi$ subject to the Cal-BCE obedience constraints \eqref{eq:bce-cond}.
\end{proposition}

\begin{proof}{Proof}
Condition on a realization of $\theta$. Given $\theta$, the effective signal $\tilde\omega$ defines a standard incomplete-information game whose payoffs are known to the agents, with information structure $\tilde\zeta(\tilde\omega\mid\gamma,\theta)$ in \eqref{eq:tilde-zeta}. By the revelation principle for Bayes correlated equilibrium \cite{Bergemann2013,Bergemann2019}, this game admits an outcome-equivalent representation in which signals are action recommendations and each agent's recommendation is obedient; i.e.\ one may take $\Omega_i=\mathcal{A}_i$ and $\tilde\omega_i$ to be the recommended action. Since $\alpha_i$ is $\theta$-measurable and $\lambda$ is common knowledge, the pair $(a_i,\theta)$ with $a_i=\tilde\omega_i+\alpha_i$ is informationally equivalent to $(\tilde\omega_i,\theta)$ for agent $i$; hence the corrected-recommendation mechanism reproduces the action distribution $\phi$ induced by $\tilde\zeta$, and obedience of $\tilde\omega_i$ is exactly inequality \eqref{eq:bce-cond} evaluated at $\theta$. Taking expectations over $\theta\sim\xi$ (reflecting the designer's uncertainty) yields the Cal-BCE obedience condition as stated. Since the designer's objective depends on $(\zeta,\lambda)$ only through the induced $\phi$, the design problem may be posed directly over $\phi$ subject to \eqref{eq:bce-cond}. \qed
\end{proof}

\begin{remark}[Realizability: optimizing over action distributions]
Denote by $C(\mathcal{Z}) = \{\phi : \phi \text{ satisfies } \eqref{eq_phi} \text{ for some } \zeta \in \mathcal{Z}\}$
the set of all action distributions that can be induced by admissible information structures.
In what follows, the design is posed directly over $\phi$ (equivalently, its covariance block $X$) rather than over $\zeta$: in the jointly Gaussian LQG setting, any $X\in\mathbb{S}_+^{2n}$ whose lower-right block equals the prescribed $\mathrm{var}(\gamma)$ is the covariance of some jointly Gaussian $(a,\gamma)$, and is therefore induced by a feasible Gaussian information structure $\zeta$. This realizability is what justifies replacing the design over $\zeta$ by a semidefinite program over $X$.
\end{remark}

The CVaR and probabilistic information design problems are stated next: 

\begin{problem}[CVAR INFORMATION DESIGN OVER ACTION DISTRIBUTION] \label{prob_cvar_phi}
Conditional-Value-at-Risk is defined formally below.
\begin{definition}[Conditional-Value-at-Risk]
    Conditional-Value-at-Risk (CVaR) for a decision vector $x$ and random vector $y$  at level $\beta \in [0,1]$ is defined as follows \cite{rockafellar2000optimization}:
\begin{equation}
CVaR(q(x,y);\beta)=(1-\beta)^{-1} \int_{q(x,y)> \rho} q(x,y)p(y)dy 
\end{equation}
 where $q(x,y)$ is loss function, $p(y)$ denotes the probability density of $y$ and $\rho$ denotes Value-at-Risk (VaR) which is defined as following:
\begin{align}\label{eq_var_def}
    VaR(q(x,y);\beta) = \inf \bigg\{\varphi\bigg|\; \int_{q(x,y)\leq \varphi}p(y)d(y) \geq \beta \bigg\}.
\end{align}
\end{definition}
Then, CVaR information design problem is as follows:
\begin{align}\label{eq:orig-obj_cvar_phi}
&\max_{\phi,\;t}\;t \\
&\text{s.t } CVaR_{\xi}\bigg(t-E_{\phi}[f(a, \gamma,\theta)];\beta_{F}\bigg)\leq 0 \label{eq:orig-obj_cvar_phi2} \\ &CVaR_{\xi}\bigg(E_{\phi}[u_{i}((a_{i}-\alpha_{i},a_{-i}),\gamma,\theta)|a_{i}]; \beta_{i}\bigg) \geq CVaR_{\xi}\bigg(E_{\phi}[u_{i}((a'_{i},a_{-i}),\gamma,\theta)|a_{i}]; \beta_{i} \bigg), \nonumber \\ &\forall a_{i}, a'_{i} \in \mathcal{A}_{i}, i\in \mathcal{N}\label{eq:bce-cond_cvar}
\end{align}
where \eqref{eq:orig-obj_cvar_phi}-\eqref{eq:orig-obj_cvar_phi2} denote the objective and \eqref{eq:bce-cond_cvar} represents the equilibrium constraint which means agents will obtain lower CVaR of expected utility if they deviate from equilibrium action $a_{i}, \forall i \in \mathcal{N}$. 
\end{problem}

\begin{problem}[PROBABILISTIC INF. DESIGN OVER ACTION DISTRIBUTION]\label{prob_chance_phi} 
The probabilistic information design problem is as follows: 
\begin{align}\label{eq:orig-obj_chance_phi}
&\max_{\phi, \;t} \;t\\ &\text{s.t. }\text{Pr}_{\xi}\bigg(E_{\phi}[f(a, \gamma,\theta)]\geq t\bigg)\geq \beta_{F} \label{eq:orig-obj_chance_phi2} \\
 & \text{Pr}_{\xi}\bigg(E_{\phi}[u_{i}((a_{i}-\alpha_{i},a_{-i}),\gamma,\theta)|a_{i}] \geq E_{\phi}[u_{i}((a'_{i},a_{-i}),\gamma,\theta)|a_{i}] \bigg)\geq \beta_{i}, \forall a_{i}, a'_{i} \in \mathcal{A}_{i},  \; i\in \mathcal{N} \label{eq:bce-cond_chance}
\end{align}
where \eqref{eq:orig-obj_chance_phi}-\eqref{eq:orig-obj_chance_phi2} denote the objective and \eqref{eq:bce-cond_chance} represents the equilibrium constraint which means BCE condition \eqref{eq:bce-cond} will hold for all $\theta$ values up to $\beta^{th}$ percentile of distribution $\xi.$
\end{problem}

Both problems \ref{prob_cvar_phi} and \ref{prob_chance_phi} are non-trivial in current forms.  The analysis specializes to quadratic utilities, Gaussian payoff states and information structures to obtain them as convex programs, namely SDPs and SOCPs. 



\subsection{Information Design Objectives}
The choice of objective function plays a central role in information design, as it determines what aspect of system performance the designer seeks to optimize under informational constraints. Depending on the application, the objective may emphasize aggregate efficiency, systemic risk mitigation, or stability of agents’ actions, reflecting the diverse purposes that information design can serve.
\begin{example}[Social Welfare] \label{ex_social_welfare}
Social welfare is the sum of individual utility functions, 
\begin{align}\label{eq_soc_welfare}
f_{welfare}(a, \gamma, \theta) &= \sum_{i=1}^{n} u_{i}(a,\gamma,\theta).
\end{align}
Social welfare is a common design objective used in congestion  \cite{brown2017studies,wu2021value}, global \cite{morris2002social}  or public goods games \cite{alizamir2020healthcare}.
\end{example}

\begin{example}[Contagion Suppression]\label{ex_contagion}
In systems with cross-sector interdependencies, a key source of systemic risk is the co-movement of agents' actions: when agents react to shocks in a correlated manner, losses can amplify across the network.
A designer wishing to suppress this form of contagion may minimize the total cross-agent action covariance:
\begin{align}\label{eq_obj_contagion}
f_{\text{contagion}}(a) &= -\!\sum_{i \neq j} \mathrm{Cov}(a_i,\, a_j).
\end{align}
In terms of the covariance variable $X$, this objective takes the form $F_{\text{contagion}}\bullet X$ where $[F_{\text{contagion}}]_{1,1} = -(\mathbf{1}\mathbf{1}^\top - I)$ and all other blocks are zero, so that
\[
F_{\text{contagion}}\bullet X \;=\; -\!\sum_{i\neq j} X_{ij}.
\]
Maximizing this objective drives off-diagonal action covariances toward zero, decoupling each agent's response to payoff shocks from those of other agents and thereby reducing the potential for shock amplification across the network.
\end{example}


\paragraph{Discussion.}
The two examples illustrate that the information design framework is flexible with respect to the choice of objective.
Social welfare (Example~\ref{ex_social_welfare}) captures aggregate efficiency; contagion suppression (Example~\ref{ex_contagion}) directly penalizes cross-agent co-movement in the action covariance.
Depending on the application---whether in financial markets, congestion games, or public goods provision---any of these objectives can be embedded in the probabilistic or CVaR information design models developed in this paper.

\section{Information Design in LQG Games}\label{sec:lqg}
An LQG game corresponds to an incomplete information game with quadratic payoff functions and Gaussian information structures and payoff states. Specifically, each player $i \in \mathcal{N}$ decides on his action $a_{i} \in \mathcal{A}_{i} \equiv \mathbb{R}$ according to a payoff function
\begin{equation}\label{utility}
u_{i}(a , \gamma ) = - H_{i,i}a_{i}^{2} - 2 \sum_{j \neq i}H_{i,j}a_{i}a_{j} + 2\gamma_{i}a_{i} +d_{i}(a_{-i}, \gamma)
\end{equation}
where $ \mathcal{A} \equiv \mathbb{R}^{n} $ and $\Gamma \equiv \mathbb{R}^{n}$ that is a quadratic function of player $i$'s action, and is bilinear with respect to $a_i$ and $a_j$, and $a_i$ and $\gamma$. The term  $d_{i}(a_{-i},\gamma)$ is an arbitrary function of the opponents' actions $a_{-i}\equiv (a_{j})_{j \neq i} $ and payoff state $\gamma$; since it does not depend on $a_i$, it does not affect player $i$'s best response. I collect the coefficients of the quadratic payoff function in a matrix $H = [H_{i,j}]_{n\times n}$ which corresponds to $\Theta$ in game $G.$

Payoff state $\gamma$ follows a Gaussian distribution, i.e., $\gamma \sim \psi(\mu , \Sigma)$ where $\psi$ is a multivariate normal probability distribution with mean $\mu\in \mathbb{R}^n$ and covariance matrix $\Sigma$. Payoff structure of the game is defined as $T \equiv ((u_{i})_{i \in N}, \psi)$. Throughout, I adopt the \emph{centering} convention $\mu=\mathbb{E}[\gamma]=0$ (obtained by reparameterizing the prior around its mean, i.e.\ demeaning), together with a mean-zero corrector $\mathbb{E}[\alpha]=0$. This is the natural frame for information design, which reshapes the second-moment (belief and co-movement) structure of $(a,\gamma)$ rather than the exogenous unconditional means; under it the obedience first-moment balance holds with $\mu_a=0$, as detailed in Remark~\ref{rem:first-moment} under Appendix \ref{app_mean} and the design reduces without loss to a program over covariances. The general case $\mu\neq0$ is also treated there.
Each player $i \in \mathcal{N}$ receives a private signal $\omega_{i} \in \Omega_{i} \equiv \mathbb{R}^{m_{i}}$ for some $m_i\in \mathbb{N}^+$. I define the information structure of the game  $\zeta(\omega|\gamma)$ as the conditional distribution of $ \omega\equiv (\omega_{i})_{i \in N}$ given $\gamma$.  I assume the joint distribution over the random variables $(\omega,\gamma)$ is Gaussian; thus, $\zeta$ is a Gaussian distribution.

To illustrate the framework, a running example of an LQG investment game with sector interactions is presented next.
\begin{example}[Investment Game with Sector Interactions]\label{ex:investment}
Each agent $i$ represents an investment sector (e.g., technology, energy, financials) and chooses an investment level $a_i \in \mathbb{R}$. 
The investment entails a quadratic cost, so the direct cost to agent $i$ is proportional to $a_i^2$. 
The private benefit of investment depends on a payoff parameter $\gamma_i$, which can be interpreted as the sector-specific marginal return. 
In addition, the payoff to agent $i$ is influenced by the investment actions of other sectors $j \neq i$ through cross-impact coefficients $H_{ij}$. 

Formally, the payoff of sector $i$ is
\[
u_i(a,\gamma) \;=\; - H_{ii} a_i^2 \;-\; 2 \sum_{j \neq i} H_{ij} a_i a_j \;+\; 2 \gamma_i a_i,
\]
where $H_{ii}>0$ represents the self-cost of investment in sector $i$, 
$H_{ij}$ captures the effect of sector $j$’s investment on sector $i$, 
and $\gamma_i$ is the private marginal return. 

The marginal return to sector $i$’s action is then
\[
\frac{\partial u_i}{\partial a_i} \;=\; -2H_{ii} a_i \;-\; 2 \sum_{j \neq i} H_{ij} a_j \;+\; 2\gamma_i,
\]
which highlights the trade-off between the private return $\gamma_i$ and the aggregate interaction term 
$-2 \sum_{j \neq i} H_{ij} a_j$. 

In numerical experiments given at the section \ref{sec:experiments}, the cross-impact coefficients $H_{ij}$ are estimated empirically from sector exchange traded fund (ETF) price and volume data, 
so that the model captures both the direct investment cost and the spillovers across sectors (See also Appendix \ref{app_etf} for ETF summary tables). 
The information designer then uses this payoff structure to determine optimal information policies under uncertainty.

\medskip
\noindent\textit{Numerical illustration.}  
Consider a two-sector example with technology ($i=1$) and energy ($i=2$). 
Suppose the parameters are
\[
H = \begin{bmatrix} 1 & 0.3 \\ 0.3 & 1.2 \end{bmatrix}, 
\qquad \gamma = (1.5,\; 2.0).
\]
Here $H_{11}=1$ and $H_{22}=1.2$ indicate each sector’s self-cost, while $H_{12}=H_{21}=0.3$ reflects positive spillovers between the sectors. 
The marginal returns are
\[
\frac{\partial u_1}{\partial a_1} = -2a_1 -0.6 a_2 + 3.0, 
\qquad 
\frac{\partial u_2}{\partial a_2} = -2.4a_2 -0.6 a_1 + 4.0.
\]
Technology benefits from its own return $\gamma_1=1.5$, but its incentives are dampened if energy invests heavily (via $-0.6a_2$). 
Conversely, energy ($\gamma_2=2.0$) has higher intrinsic return but also higher self-cost ($H_{22}=1.2$). 
This toy example illustrates how cross-sector interactions, even when small, reshape equilibrium investment levels.
\end{example}

The description of the information design problem for LQG games continues here, designing covariance matrices of the action distribution $\phi(a|\gamma)$ and the corrector distribution $\lambda(\alpha|\theta)$. In LQG games the uncertain primitive is the payoff matrix $H\sim\xi$. Because the corrector acts per agent ($\alpha\in\mathbb{R}^n$), I take the \emph{uncertainty state} to be the vector of own-coefficients
\[
\theta := \mathrm{diag}(H) = (H_{11},\dots,H_{nn})\in\mathbb{R}^n, \qquad \theta_i := H_{ii},
\]
which is precisely the coefficient that multiplies the corrector $\alpha_i$ in each agent's obedience condition. The designer's uncertainty about the \emph{off-diagonal} (cross-agent) coefficients of $H$ is not absorbed by the corrector but is carried through the robust action-channel constraints below, where the full matrix $H$ enters as a random object. With this convention both blocks are $2n\times2n$, and the designer's decisions concern the covariance matrices $X \in \mathbb{S}_+^{2n}$ and $G \in \mathbb{S}_+^{2n}$ defined below:

\begin{equation}\label{eq_x_defn}
X:=\begin{bmatrix}
var(a) & cov(a, \gamma)\\
cov(\gamma, a) & var(\gamma)
\end{bmatrix} 
\text{ and } G:=\begin{bmatrix}
var(\alpha) & cov(\alpha, \theta)\\
cov(\theta, \alpha) & var(\theta)
\end{bmatrix} 
\end{equation} 
where $\mathbb{S}_+^{2n}$ denotes positive semidefinite $2n\times 2n$ matrices. Note that the submatrices $var(\gamma)$ and $var(\theta)$ are given parameters; in particular $var(\theta)=var(\mathrm{diag}(H))$ is the covariance of the own-coefficients under $\xi$.

The expected value of the objective under $\phi$ can be reformulated using the Frobenius product in terms of $X$ and coefficient matrix $F$ as follows,
\begin{align}
E_{\phi}[f(a,\gamma,\theta)] =& E_\phi\big[ \begin{bmatrix}
a^{T}, &\gamma^{T}
\end{bmatrix} F \begin{bmatrix}
a \\ \gamma
\end{bmatrix} \big ], \\
=& F \bullet X \label{eq_obj_exp}
\end{align} where 
$
F=\begin{bmatrix}
[F]_{1,1} & [F]_{1,2}\\ [F]_{1,2} & [F]_{2,2}
\end{bmatrix} \in P^{2n},
$ and note that $[F]_{i,j}$ denotes the $i,j$th $n\times n$ submatrix. the submatrix $[F]_{2,2}$ is set to a zero matrix because $var(\gamma)$ is given. Under the centering convention $\mu=0$ and mean-zero corrector, the identity \eqref{eq_obj_exp} is exact; for $\mu\neq0$ it acquires the additive term $z^\top F z$ with $z=[\mu_a;\mu]$ (see Remark~\ref{rem:first-moment} in Appendix \ref{app_mean}).

The following lemmas describe the equilibrium constraints.

\begin{lemma}[CVaR--Cal-BCE constraints for LQG with linear policies] \label{lem:linear_cvar_X}
Let $u_i(a,\gamma,H)$ be given by \eqref{utility} and suppose:
(i) $(a,\gamma)$ are jointly Gaussian under $\phi$, and $(\alpha,\theta)$ are jointly Gaussian under $\lambda$ with $\theta=\mathrm{diag}(H)$ (so $\theta_i=H_{ii}$);
(ii) $H$ is independent of $(a,\gamma)$ and of the information structure conditional on $\theta$ (designer uncertainty only);
(iii) $H_{ii}>0$ for all $i$.
Then the CVaR–Cal-BCE inequality \eqref{eq:bce-cond_cvar} implies the following convex constraints, in which the equality residuals are controlled in \emph{both} tails:
\begin{align}
&\mathrm{CVaR}_{\xi}\!\bigg(\pm\Big(\sum_{j\in\mathcal N} H_{ij}\,X_{ij}-X_{i,n+i}\Big)\;;\;\beta_i\bigg)\;\le 0,\qquad i\in\mathcal N, \label{eq_affected_inequality_app_cvar}\\
&\mathrm{CVaR}_{\xi}\!\bigg(\pm\Big(\sum_{j\in\mathcal N} H_{ij}\,G_{ij}-G_{i,n+i}\Big)\;;\;\beta_i\bigg)\;\le 0,\qquad i\in\mathcal N, \label{eq_affected_inequality_app_cvar_2}
\end{align}
where ``$\pm$'' means that both the functional and its negation are constrained, so the moment equality $\Lambda^{(\cdot)}_i(H)=0$ is enforced as $\mathrm{CVaR}_\xi(\Lambda;\beta_i)\le0$ \emph{and} $\mathrm{CVaR}_\xi(-\Lambda;\beta_i)\le0$.
Here $X\!\in\! \mathbb S_+^{2n}$ and $G\!\in\!\mathbb S_+^{2n}$ are the covariance blocks defined by \eqref{eq_x_defn}:
$X_{ij}=\mathrm{cov}(a_i,a_j)$ for $1\!\le i,j\!\le n$, and $X_{i,n+i}=\mathrm{cov}(a_i,\gamma_i)$; 
$G_{ij}=\mathrm{cov}(\alpha_i,\alpha_j)$ for $1\!\le i,j\!\le n$, and $G_{i,n+i}=\mathrm{cov}(\alpha_i,\theta_i)$.
\end{lemma}

\begin{proof}{Proof}
\emph{Step 1 (Deviation gap and conditional best response).}
Fix player $i$ and a recommended signal/action $a_i$.
For any deviation $a_i'\!\in\!\mathbb R$, the conditional expected payoff under $\phi$ is
\[
\Phi_i(a_i'\,;a_i,H)\;:=\; E_{\phi}\big[u_i\big( (a_i',a_{-i}),\gamma,H\big)\;\big|\;a_i\big]
= -H_{ii}(a_i')^2 - 2\sum_{j\ne i} H_{ij}\,a_i'\,m_{ij} + 2 a_i'\,m_{i\gamma} + C,
\]
where $m_{ij}:=E_{\phi}[a_j\,|\,a_i]$ and $m_{i\gamma}:=E_{\phi}[\gamma_i\,|\,a_i]$, and $C$ does not depend on $a_i'$.
Since $H_{ii}>0$, $\Phi_i(\cdot;a_i,H)$ is a strictly concave quadratic in $a_i'$.
Hence the unique conditional best response given $a_i$ is
\begin{equation}\label{eq:BRi}
a_i^{\star}(a_i;H)\;=\;\frac{m_{i\gamma}-\sum_{j\ne i} H_{ij}\,m_{ij}}{H_{ii}}.
\end{equation}
The Cal-BCE (with corrector) requires that the played action $(a_i-\alpha_i)$ weakly dominates any deviation $a_i'$ \emph{given} $a_i$; for a concave objective this is equivalent to requiring that $(a_i-\alpha_i)$ coincides with the maximizer in \eqref{eq:BRi}, almost surely:
\begin{equation}\label{eq:FOC_pathwise}
H_{ii}\,(a_i-\alpha_i)\;+\;\sum_{j\ne i} H_{ij}\,m_{ij}\;-\;m_{i\gamma}\;=\;0\quad\text{a.s.}
\end{equation}

\emph{Step 2 (From pathwise optimality to moment conditions).}
Equation \eqref{eq:FOC_pathwise} is an affine relation in the random objects $(a,\gamma,\alpha)$.
To obtain constraints that only involve second moments (the SDP variables), multiply \eqref{eq:FOC_pathwise} by a centered instrument and take expectations, using the law of iterated expectations.

\smallskip
\noindent\underline{(a) BCE channel (actions vs. state).}
Multiply \eqref{eq:FOC_pathwise} by $(a_i-E[a_i])$ and take $E_\phi[\cdot]$.
Using $E[(a_i-E[a_i])\,m_{ij}]=\mathrm{cov}(a_i,a_j)$ and $E[(a_i-E[a_i])\,m_{i\gamma}]=\mathrm{cov}(a_i,\gamma_i)$, I get
\[
H_{ii}\,\mathrm{var}(a_i)\;+\;\sum_{j\ne i} H_{ij}\,\mathrm{cov}(a_i,a_j)\;-\;\mathrm{cov}(a_i,\gamma_i)\;=\;0,
\]
i.e.,
\begin{equation}\label{eq:BCE_moment}
\sum_{j\in\mathcal N} H_{ij}\,X_{ij}\;-\;X_{i,n+i}\;=\;0,
\end{equation}
with $X_{ij}=\mathrm{cov}(a_i,a_j)$ and $X_{i,n+i}=\mathrm{cov}(a_i,\gamma_i)$, as defined in \eqref{eq_x_defn}.
Thus, the \emph{no-deviation condition in expectation} along the action/state channel is a single linear moment equation in the $X$-block.

\smallskip
\noindent\underline{(b) Corrector channel (designer’s uncertainty).}
Because the designer does not know $\theta=\mathrm{diag}(H)$ while agents do, the corrector $\alpha$ is drawn from $\lambda(\alpha\,|\,\theta)$ and enters the played action $(a_i-\alpha_i)$.
Repeat the same argument, now multiplying \eqref{eq:FOC_pathwise} by $(\alpha_i-E[\alpha_i])$ and taking expectation over $\lambda$ (and $\phi$ when needed). Under the structural assumptions stated above—$\alpha$ measurable with respect to $\theta$, and independent of $(\gamma,\omega)$ given $\theta$—I obtain
\[
\sum_{j\in\mathcal N} H_{ij}\,\mathrm{cov}(\alpha_i,\alpha_j)\;-\;\mathrm{cov}(\alpha_i,\theta_i)\;=\;0,
\]
i.e.,
\begin{equation}\label{eq:Cal-BCE_moment}
\sum_{j\in\mathcal N} H_{ij}\,G_{ij}\;-\;G_{i,n+i}\;=\;0,
\end{equation}
with $G_{ij}=\mathrm{cov}(\alpha_i,\alpha_j)$ and $G_{i,n+i}=\mathrm{cov}(\alpha_i,\theta_i)$, matching the $G$-block in \eqref{eq_x_defn}.
(Informally: the corrector aligns the uncertain coefficient channel, so its second moments must satisfy the analogous no-deviation moment balance.)

\emph{Step 3 (From equalities to CVaR constraints).}
The random matrix $H$ is drawn from $\xi$.
Both left-hand sides in \eqref{eq:BCE_moment}–\eqref{eq:Cal-BCE_moment} are \emph{affine} in $H$:
\[
\Lambda^{(X)}_i(H)\;:=\;\sum_j H_{ij} X_{ij} - X_{i,n+i},\qquad
\Lambda^{(G)}_i(H)\;:=\;\sum_j H_{ij} G_{ij} - G_{i,n+i}.
\]
By coherence and monotonicity of CVaR, enforcing
$\mathrm{CVaR}_{\xi}(\Lambda^{(X)}_i(H);\beta_i)\le 0$
and
$\mathrm{CVaR}_{\xi}(\Lambda^{(G)}_i(H);\beta_i)\le 0$
ensures that the \emph{worst $\,(1-\beta_i)$-tail} (over draws of $H$) respects the no-deviation moment balance in both channels. This yields \eqref{eq_affected_inequality_app_cvar}–\eqref{eq_affected_inequality_app_cvar_2}.

\emph{Step 4 (Why it suffices to check only the best deviation).}
Since $\Phi_i(\cdot;a_i,H)$ is strictly concave in $a_i'$, the inequality
\[
E_{\phi}\!\left[u_i\big((a_i-\alpha_i,a_{-i}),\gamma,H\big)\,\big|\,a_i\right]
\;\ge\;
E_{\phi}\!\left[u_i\big((a_i',a_{-i}),\gamma,H\big)\,\big|\,a_i\right]
\quad \forall a_i'
\]
holds if and only if it holds at the unique maximizer $a_i^{\star}(a_i;H)$ in \eqref{eq:BRi}, which is equivalent to the first-order condition \eqref{eq:FOC_pathwise}.
Thus the deviation check reduces to the first-order condition \eqref{eq:FOC_pathwise}. Projecting \eqref{eq:FOC_pathwise} onto the centered instruments in Step~2 yields the second-moment identities \eqref{eq:BCE_moment}--\eqref{eq:Cal-BCE_moment} as \emph{necessary} consequences, and enforcing their CVaR counterparts in Step~3 gives the stated convex constraints. \qed
\end{proof}

\begin{lemma}[Chance-constrained Cal-BCE for LQG with linear policies]\label{lem_prob}
Let $u_i(a,\gamma,H)$ be given by \eqref{utility} and suppose:
(i) $(a,\gamma)$ are jointly Gaussian under $\phi$, and $(\alpha,\theta)$ are jointly Gaussian under $\lambda$ with $\theta=\mathrm{diag}(H)$ (so $\theta_i=H_{ii}$);
(ii) $H$ is independent of $(a,\gamma)$ and of the information structure conditional on $\theta$;
(iii) $H_{ii}>0$ for all $i$; and
(iv) the distribution $\xi$ of $H$ is absolutely continuous.
Define the affine functionals
\[
\Lambda^{(X)}_i(H):=\sum_{j\in\mathcal N} H_{ij}\,X_{ij}-X_{i,n+i},\qquad
\Lambda^{(G)}_i(H):=\sum_{j\in\mathcal N} H_{ij}\,G_{ij}-G_{i,n+i},
\]
where the covariance blocks $X,G$ are as in \eqref{eq_x_defn}.
Then the chance-constrained Cal-BCE condition \eqref{eq:bce-cond_chance} is enforced by the following \emph{two-sided} chance constraints, for some tolerances $\varepsilon_X,\varepsilon_G\ge 0$ and confidence levels $\beta_i\in(0,1)$:
\begin{align}
Pr_{\xi}\!\big(|\Lambda^{(X)}_i(H)|\le \varepsilon_X\big)\;\ge\;\beta_i,\qquad
Pr_{\xi}\!\big(|\Lambda^{(G)}_i(H)|\le \varepsilon_G\big)\;\ge\;\beta_i,\qquad i\in\mathcal N, \label{eq:chance_bands}
\end{align}
are sufficient. In the limit $\varepsilon_X=\varepsilon_G=0$, the two-sided bands reduce to equalities, but $Pr_{\xi}(\Lambda=0)=0$ under (iv), so a nonzero tolerance is required for nontrivial feasibility in the probabilistic model.
\end{lemma}

\begin{proof}{Proof}
\emph{Step 1 (Conditional best response is linear in conditional means).}
Fix $i$ and a recommended $a_i$. For any deviation $a'_i\in \mathbb{R}$,
\[
E_{\phi}\!\big[u_i\big((a'_i,a_{-i}),\gamma,H\big)\big|a_i\big]
= -H_{ii}(a'_i)^2 - 2\sum_{j\neq i} H_{ij} a'_i\,m_{ij} + 2 a'_i\,m_{i\gamma} + C,
\]
with $m_{ij}:=E_{\phi}[a_j|a_i]$, $m_{i\gamma}:=E_{\phi}[\gamma_i|a_i]$ and constant $C$ independent of $a'_i$.
Since $H_{ii}>0$, the unique conditional maximizer is
\begin{equation}\label{eq:BRi_prob}
a_i^{\star}(a_i;H)=\frac{m_{i\gamma}-\sum_{j\neq i} H_{ij}\,m_{ij}}{H_{ii}}.
\end{equation}

\emph{Step 2 (Pathwise optimality implies a linear identity).}
The Cal-BCE condition in the chance model requires that the played action $(a_i-\alpha_i)$ be (weakly) optimal with high probability over $H$. Concavity implies it suffices to check the unique maximizer, hence
\begin{equation}\label{eq:pathwise_prob}
H_{ii}\,(a_i-\alpha_i)+\sum_{j\neq i} H_{ij}\,m_{ij}-m_{i\gamma}=0
\quad \text{holds with high probability over }H.
\end{equation}

\emph{Step 3 (Instrument with centered variables to obtain moment identities).}
Multiply \eqref{eq:pathwise_prob} by $(a_i-E[a_i])$ and take unconditional expectations under $\phi$; using the law of iterated expectations and $E[(a_i-E[a_i])\,m_{ij}]=Cov(a_i,a_j)$ and $E[(a_i-E[a_i])\,m_{i\gamma}]=Cov(a_i,\gamma_i)$ yields
\[
\sum_{j\in\mathcal N} H_{ij}\,X_{ij}-X_{i,n+i}=0.
\]
Likewise, multiplying \eqref{eq:pathwise_prob} by $(\alpha_i-E[\alpha_i])$ and taking expectations under $\lambda$ (and $\phi$ when needed), under the structural independence in the statement, gives
\[
\sum_{j\in\mathcal N} H_{ij}\,G_{ij}-G_{i,n+i}=0.
\]
Thus, for each $i$ the affine functionals $\Lambda^{(X)}_i(H)$ and $\Lambda^{(G)}_i(H)$ are obtained.

\emph{Step 4 (Chance constraints).}
Because $H\sim\xi$ and $\Lambda^{(\cdot)}_i(H)$ are affine in $H$, enforcing that the identity in Step 3 holds \emph{with probability at least $\beta_i$} leads to the two-sided probabilistic band in \eqref{eq:chance_bands}.

Under (iv), the exact equality event has measure zero, hence $\varepsilon_X=\varepsilon_G=0$ renders the chance constraints trivial unless $\beta_i=0$; in practice one takes small nonzero tolerances. \qed
\end{proof}

\paragraph{Discussion.}
Lemmas~\ref{lem:linear_cvar_X} and~\ref{lem_prob} establish that the Calibrated Bayes Correlated Equilibrium (Cal-BCE) conditions can be represented as convex semidefinite constraints on the covariance matrices $X$ and $G$ in linear–quadratic–Gaussian (LQG) games. The $X$-block enforces the Bayes correlated equilibrium (BCE) consistency between actions and payoff states, while the $G$-block captures the corrective alignment that accounts for the designer’s uncertainty about the payoff matrix $H$.  
Together, they ensure that the information designer’s policy induces strategies that remain incentive-compatible even when payoffs are only partially known.

The probabilistic and CVaR formulations correspond to distinct attitudes toward uncertainty. The \emph{probabilistic Cal-BCE} enforces equilibrium feasibility with high probability under the distribution of $H$, ensuring that best-response relations hold “most of the time” while allowing small deviations in rare cases.  
The \emph{CVaR–Cal-BCE}, by contrast, guarantees that the same equilibrium balance holds in expectation over the worst tail of the distribution, eliminating profitable deviations even under extreme parameter realizations.  
Hence, the probabilistic model emphasizes efficiency and average-case performance, whereas the CVaR model prioritizes robustness and stability against systemic shocks.

Economically, these formulations provide two complementary risk postures for information design under parameter uncertainty.  
The probabilistic version corresponds to a designer optimizing for expected welfare when the environment is stable, while the CVaR formulation represents a risk-averse designer seeking to suppress destabilizing strategic interdependence and preserve equilibrium stability under stress.  
Cal-BCE thus generalizes the classical BCE by embedding a risk-sensitive correction mechanism that ensures equilibrium consistency under payoff uncertainty.  
When applied to the linear–quadratic–Gaussian (LQG) setting, this framework admits a convex representation that enables tractable computation and comparative analysis of efficiency–robustness trade-offs.

Detailed descriptions of the CVaR (Section~\ref{sec_cvar}) and probabilistic (Section~\ref{sec_chance}) information design models in LQG games. 

\section{Information Design under Risk with CVaR}\label{sec_cvar}
The convex programming formulation of the CVaR information design model in Problem~\ref{prob_cvar_phi} follows. First provide the following well-known representation of Conditional Value-at-Risk due to Rockafellar and Uryasev \cite{rockafellar2000optimization}.

\begin{lemma}[Rockafellar--Uryasev representation of CVaR]\label{lem:RU}
For any random variable $Z$ and confidence level $\beta \in (0,1)$,
\[
\mathrm{CVaR}_{\xi}(Z;\beta)
= \min_{\rho \in \mathbb{R}} \Big\{ \rho + \tfrac{1}{1-\beta} \, \mathbb{E}_{\xi}\big[(Z-\rho)_{+}\big] \Big\},
\]
where $(x)_{+}=\max\{x,0\}$. 
\end{lemma}

\noindent
This lemma allows each CVaR constraint to be expressed in terms of an auxiliary variable $\rho$ and linear inequalities involving slack variables. I also state the Sample Average Approximation of CVaR.

\begin{lemma}[Sample Average Approximation of CVaR]\label{lem:SAA}
Let $Z(H)$ be a random variable depending on random payoff matrix $H \sim \xi$,
and let $\{H_j\}_{j=1}^J$ be i.i.d.\ samples from $\xi$. 
For any $\beta \in (0,1)$, the CVaR constraint
\[
\mathrm{CVaR}_{\xi}(Z(H);\beta) \leq 0
\]
is approximated by the finite-sample program
\[
\exists \,\rho \in \mathbb{R}, \; \{z_j \ge 0\}_{j=1}^J \quad \text{s.t.} \quad 
z_j \ge Z(H_j)-\rho,\;\; j=1,\dots,J,
\]
\[
\rho + \nu \sum_{j=1}^J z_j \leq 0,
\]
where $\nu = \tfrac{1}{(1-\beta)J}$. 
As $J \to \infty$, this sample-average approximation converges almost surely to the exact CVaR constraint. 
\end{lemma}

The following matrices are also used in the definitions of the CVaR and probabilistic models: $M_{k,l}= [[M_{k,l}]_{i,j}]_{2nx2n} \in P^{2n},  k \in \mathcal{N}$ is given as: 
\begin{equation}\label{eq_M_matrix}
[M_{k,l}]_{i,j}=\begin{cases}
1/2 \quad\text{ if } k<l, i=n+k, j=n+l\\
1/2 \quad\text{ if } k<l, i=n+l, j=n+k\\
1 \quad\text{ if } k=l, i=n+k, j=n+l\\
0 \quad \text{otherwise,}
\end{cases} \qquad 
\end{equation}

and $R_{l}= [[R_{ l}]_{i,j}]_{2n\times 2n} \in P^{2n}, l \in \mathcal{N}$ is given as: 
\begin{equation}\label{eq_R}
[R_{l}]_{i,j}=\begin{cases}
H_{l,l}&if \quad  i = j = l,  \\
H_{l,j}/2  & if \quad  i = l, 1\leq j \leq n, j\neq l, \\
-1/2 &if \quad i = l, j= n + l,\\
H_{l,i}/2 &if\quad  j = l, 1\leq i \leq n, i\neq l \\
-1/2 &if \quad j = l, i= n + l,\\
0 & \text{otherwise.}
\end{cases}
\end{equation}
With this definition, and using the symmetry of $X$, one has $R_l\bullet X=\sum_{j\in\mathcal N}H_{l,j}X_{lj}-X_{l,n+l}=\Lambda^{(X)}_l(H)$, i.e.\ $R_l$ encodes \emph{row $l$} of $H$, which is exactly the row that enters player $l$'s payoff \eqref{utility}. In particular $H$ need not be symmetric: it is used as estimated, and the construction is faithful to the directional cross-impacts $H_{lj}\neq H_{jl}$ (only the symmetric part would matter had the $(j,l)$ entry used the column coefficient $H_{j,l}$ instead).

\begin{theorem}\label{theorem_cvar}
The CVaR information design problem \eqref{eq:orig-obj_cvar_phi}--\eqref{eq:bce-cond_cvar} can be written as the following convex program under a finite set of $J$ sampled payoff scenarios $\{F_j, R_{l,j}\}_{j=1}^J$: 
\begin{align}
&\max_{X \in P^{2n}_{+},\;t,\;\rho_{F},\;\rho^{+}_{l},\rho^{-}_{l},\;z_F,\;z^{+}_R,\,z^{-}_R} \quad t \label{model_cvar_final}\\ 
\text{s.t. }& z_{F,j} \ge t- F_{j}\bullet X - \rho_{F}, \quad j=1,\dots,J, \label{eq_linear_f_slack}\\
& \rho_{F} + \nu_F \sum_{j=1}^{J} z_{F,j} \leq 0, \label{eq_linear_f}\\
& z^{+}_{R,j,l} \ge R_{l,j}\bullet X - \varepsilon - \rho^{+}_{l}, \quad z^{-}_{R,j,l} \ge -R_{l,j}\bullet X - \varepsilon - \rho^{-}_{l}, \quad j=1,\dots,J,\; l \in \mathcal{N}, \label{eq_linear_R_slack}\\
& \rho^{+}_{l} + \nu_l \sum_{j=1}^{J} z^{+}_{R,j,l} \leq 0, \quad \rho^{-}_{l} + \nu_l \sum_{j=1}^{J} z^{-}_{R,j,l} \leq 0, \quad l \in \mathcal{N}, \label{eq_linear_R}\\
& M_{k,l}\bullet X = \mathrm{cov}(\gamma_{k},\gamma_{l}), \quad \forall k \leq l, \; k,l \in \mathcal{N}, \label{eq_cov_gamma}
\end{align}
where $\nu_F = \tfrac{1}{(1-\beta_F)J}$ and $\nu_l = \tfrac{1}{(1-\beta_l)J}$, $z_{F,j},z^{\pm}_{R,j,l}\ge0$, and the auxiliary variables satisfy $\rho_F = \mathrm{VaR}(t-F\bullet X;\beta_F)$, $\rho^{+}_l = \mathrm{VaR}(R_l\bullet X-\varepsilon;\beta_l)$, $\rho^{-}_l = \mathrm{VaR}(-R_l\bullet X-\varepsilon;\beta_l)$. The two slack families implement $\mathrm{CVaR}_\xi(R_l\bullet X-\varepsilon;\beta_l)\le0$ and $\mathrm{CVaR}_\xi(-R_l\bullet X-\varepsilon;\beta_l)\le0$, i.e.\ the two-sided tail control of the equilibrium residual $\Lambda^{(X)}_l(H)=R_l\bullet X$.
\end{theorem}

\begin{proof}{Proof}
Begin with the CVaR constraints from Problem~\ref{prob_cvar_phi}, where the equilibrium equality residual is controlled in both tails (Lemma~\ref{lem:linear_cvar_X}):
\[
\mathrm{CVaR}_{\xi}(t-F\bullet X;\beta_F) \leq 0, 
\qquad
\mathrm{CVaR}_{\xi}\big(\pm(R_l \bullet X) - \varepsilon;\beta_l\big) \leq 0, \;\forall l \in \mathcal{N}.
\]

\paragraph{Step 1: Rockafellar--Uryasev representation.}
By Lemma~\ref{lem:RU} (Rockafellar--Uryasev \cite{rockafellar2000optimization}),
\[
\mathrm{CVaR}_{\xi}(Z;\beta) = \min_{\rho \in \mathbb{R}} \Big\{ \rho + \tfrac{1}{1-\beta}\, \mathbb{E}_{\xi}[(Z-\rho)_+]\Big\}.
\]
Applying this to $Z = t-F\bullet X$ and to $Z = \pm(R_l\bullet X) - \varepsilon$ yields equivalent constraints of the form
\[
\rho_F + \tfrac{1}{1-\beta_F}\, \mathbb{E}_{\xi}\big[(t-F\bullet X-\rho_F)_+\big] \leq 0,
\]
\[
\rho^{+}_l + \tfrac{1}{1-\beta_l}\, \mathbb{E}_{\xi}\big[(R_l\bullet X-\varepsilon-\rho^{+}_l)_+\big] \leq 0, 
\quad
\rho^{-}_l + \tfrac{1}{1-\beta_l}\, \mathbb{E}_{\xi}\big[(-R_l\bullet X-\varepsilon-\rho^{-}_l)_+\big] \leq 0, 
\quad l \in \mathcal{N}.
\]

\paragraph{Step 2: Sample Average Approximation.}
By Lemma~\ref{lem:SAA} (SAA of CVaR), replacing expectations by sample averages over $\{H_j\}_{j=1}^J$ gives:
\[
\rho_F + \nu_F \sum_{j=1}^J z_{F,j} \leq 0, 
\quad z_{F,j} \ge t-F_j\bullet X-\rho_F, \;\; z_{F,j}\ge 0,
\]
\[
\rho^{+}_l + \nu_l \sum_{j=1}^J z^{+}_{R,j,l} \leq 0, 
\quad z^{+}_{R,j,l} \ge R_{l,j}\bullet X-\varepsilon-\rho^{+}_l, \;\; z^{+}_{R,j,l}\ge 0,
\;\; \forall l \in \mathcal{N},
\]
\[
\rho^{-}_l + \nu_l \sum_{j=1}^J z^{-}_{R,j,l} \leq 0, 
\quad z^{-}_{R,j,l} \ge -R_{l,j}\bullet X-\varepsilon-\rho^{-}_l, \;\; z^{-}_{R,j,l}\ge 0,
\;\; \forall l \in \mathcal{N}.
\]

\paragraph{Step 3: Covariance consistency.}
Finally, to ensure the action distribution is consistent with the exogenous payoff state distribution, the following is imposed:
\[
M_{k,l}\bullet X = \mathrm{cov}(\gamma_k,\gamma_l), 
\quad \forall k \leq l,\; k,l \in \mathcal{N}.
\]

Collecting these constraints together gives the convex program
\eqref{model_cvar_final}--\eqref{eq_cov_gamma}. 
\qed
\end{proof}

\section{Chance-Constrained Information Design}\label{sec_chance}
The second-order cone (convex) reformulation of the chance-constrained information design problem defined in Problem~\ref{prob_chance_phi}.

\begin{theorem}\label{theorem_chance}
Suppose the payoff matrix $H$ is Gaussian,
\(
\mathrm{vec}(H) \sim \mathcal{N}(\mathrm{vec}(\bar H),\Sigma_H).
\)
Then the chance-constrained information design problem 
\eqref{eq:orig-obj_chance_phi}--\eqref{eq:bce-cond_chance}
is equivalent to the following second-order cone program (SOCP):
\begin{align}
&\max_{X \in P^{2n}_{+},\;t} \quad t \label{model_chance_socp}\\
\text{s.t. } \quad 
& F_\mu \bullet X - \Phi^{-1}(\beta_F)\,
  \Big\| \Sigma_H^{1/2} A_F^T \mathrm{vec}(X) \Big\|_2 
  \;\;\geq t, \label{eq_chance_obj_soc}\\
& R_{\mu,l} \bullet X + \Phi^{-1}\!\big(\tfrac{1+\beta_l}{2}\big)\,
  \Big\| \Sigma_H^{1/2} A_{R,l}^T \mathrm{vec}(X) \Big\|_2 
  \;\;\leq \varepsilon, 
  \quad l \in \mathcal{N}, \label{eq_chance_eq_soc}\\
& R_{\mu,l} \bullet X - \Phi^{-1}\!\big(\tfrac{1+\beta_l}{2}\big)\,
  \Big\| \Sigma_H^{1/2} A_{R,l}^T \mathrm{vec}(X) \Big\|_2 
  \;\;\geq -\varepsilon, 
  \quad l \in \mathcal{N}, \label{eq_chance_eq_soc_lower}\\
& M_{k,l}\bullet X = \mathrm{cov}(\gamma_k,\gamma_l), 
  \quad \forall k \leq l,\; k,l \in \mathcal{N}. \label{eq_chance_cov_soc}
\end{align}
Here $F_\mu = \mathbb{E}[F]$, $R_{\mu,l} = \mathbb{E}[R_l]$, 
and $A_F, A_{R,l}$ are the linear maps such that 
$\mathrm{vec}(F) = A_F \mathrm{vec}(H)+c_F$ and 
$\mathrm{vec}(R_l) = A_{R,l} \mathrm{vec}(H)+c_{R,l}$. Constraints \eqref{eq_chance_eq_soc}--\eqref{eq_chance_eq_soc_lower} are the two halves of the two-sided band $Pr_\xi(|\Lambda^{(X)}_l(H)|\le\varepsilon)\ge\beta_l$ of Lemma~\ref{lem_prob}: each half is enforced at confidence $\tfrac{1+\beta_l}{2}$, so by the union bound their conjunction guarantees the band at the prescribed level $\beta_l$ (with the one-sided choice $\Phi^{-1}(\beta_l)$ the band would instead hold at level $2\beta_l-1$).
\end{theorem}

\begin{proof}{Proof}
Start with the chance constraints; the equilibrium condition is the two-sided band of Lemma~\ref{lem_prob}:
\[
Pr_{\xi}(F\bullet X \ge t) \ge \beta_F, 
\qquad 
Pr_{\xi}\big(|R_l\bullet X| \le \varepsilon\big) \ge \beta_l, 
\;\; \forall l \in \mathcal{N}.
\]
By the union bound, $Pr_\xi(|R_l\bullet X|\le\varepsilon)\ge\beta_l$ is implied by the pair $Pr_\xi(R_l\bullet X\le\varepsilon)\ge\tfrac{1+\beta_l}{2}$ and $Pr_\xi(R_l\bullet X\ge-\varepsilon)\ge\tfrac{1+\beta_l}{2}$, which I reformulate below.

\paragraph{Step 1: Gaussian reformulation.}
For $Z \sim \mathcal{N}(\mu,\sigma^2)$, the inequalities 
$Pr(Z \ge t) \ge \beta$, $Pr(Z \le \varepsilon) \ge \beta'$, and $Pr(Z \ge -\varepsilon) \ge \beta'$
are equivalent to
\[
\mu - \Phi^{-1}(\beta)\sigma \ge t,
\qquad
\mu + \Phi^{-1}(\beta')\sigma \le \varepsilon,
\qquad
\mu - \Phi^{-1}(\beta')\sigma \ge -\varepsilon,
\]
where $\Phi^{-1}(\cdot)$ is the quantile function of the standard normal distribution and $\beta'=\tfrac{1+\beta_l}{2}$ for the equilibrium constraints.

\paragraph{Step 2: Variance representation.}
Since $F\bullet X$ and $R_l\bullet X$ are affine in $\mathrm{vec}(H)$, 
there exist matrices $A_F, A_{R,l}$ such that
\[
F\bullet X = (\mathrm{vec}(X))^T A_F \,\mathrm{vec}(H) + (\mathrm{vec}(X))^T c_F,
\]
\[
R_l\bullet X = (\mathrm{vec}(X))^T A_{R,l} \,\mathrm{vec}(H) + (\mathrm{vec}(X))^T c_{R,l}.
\]
Defining $w_F = A_F^T \mathrm{vec}(X)$ and $w_{R,l} = A_{R,l}^T \mathrm{vec}(X)$,
the variances are
\[
\mathrm{Var}[F\bullet X] = w_F^T \Sigma_H w_F = \big\|\Sigma_H^{1/2}w_F\big\|_2^2,
\]
\[
\mathrm{Var}[R_l\bullet X] = w_{R,l}^T \Sigma_H w_{R,l} = \big\|\Sigma_H^{1/2}w_{R,l}\big\|_2^2.
\]

\paragraph{Step 3: Substitution into chance constraints.}
The objective constraint becomes
\[
F_\mu \bullet X - \Phi^{-1}(\beta_F)\, \big\|\Sigma_H^{1/2} A_F^T \mathrm{vec}(X)\big\|_2 \ge t,
\]
and the equilibrium constraints become the two SOC halves
\[
R_{\mu,l} \bullet X + \Phi^{-1}\!\big(\tfrac{1+\beta_l}{2}\big)\, \big\|\Sigma_H^{1/2} A_{R,l}^T \mathrm{vec}(X)\big\|_2 \le \varepsilon l \in \mathcal{N}, 
\]
\[
R_{\mu,l} \bullet X - \Phi^{-1}\!\big(\tfrac{1+\beta_l}{2}\big)\, \big\|\Sigma_H^{1/2} A_{R,l}^T \mathrm{vec}(X)\big\|_2 \ge -\varepsilon, 
\quad l \in \mathcal{N}.
\]

\paragraph{Step 4: Covariance consistency.}
Finally, the same consistency condition as in the CVaR model is enforced:
\[
M_{k,l}\bullet X = \mathrm{cov}(\gamma_k,\gamma_l), \quad \forall k \le l.
\]

The resulting constraints \eqref{eq_chance_obj_soc}--\eqref{eq_chance_cov_soc} are second-order cone constraints, hence convex. 
\qed
\end{proof}

\begin{corollary}[Independent payoff coefficients]\label{cor:independent}
Suppose that the entries of $H$ are independent, i.e. 
\(
\Sigma_H = \mathrm{diag}(\{\sigma_{ij}^2\}_{i,j=1}^n).
\)
Then the variance terms in Theorem~\ref{theorem_chance} reduce to weighted $\ell_2$ norms:
\[
\mathrm{Var}[F\bullet X] 
= \sum_{i,j} \sigma_{ij}^2 \,\big((w_F)_{ij}\big)^2,
\qquad
\mathrm{Var}[R_l\bullet X] 
= \sum_{i,j} \sigma_{ij}^2 \,\big((w_{R,l})_{ij}\big)^2,
\]
where $w_F = A_F^T \mathrm{vec}(X)$ and $w_{R,l} = A_{R,l}^T \mathrm{vec}(X)$.
Equivalently,
\[
\big\|\Sigma_H^{1/2} w_F\big\|_2 
= \Big(\sum_{i,j} \sigma_{ij}^2 (w_F)_{ij}^2\Big)^{1/2}, 
\quad
\big\|\Sigma_H^{1/2} w_{R,l}\big\|_2 
= \Big(\sum_{i,j} \sigma_{ij}^2 (w_{R,l})_{ij}^2\Big)^{1/2}.
\]

Thus the SOCP reformulation of the chance-constrained problem 
takes the explicit form
\begin{align*}
&F_\mu \bullet X - \Phi^{-1}(\beta_F)
 \Big(\sum_{i,j} \sigma_{ij}^2 (w_F)_{ij}^2\Big)^{1/2} \;\;\geq t, \\
&R_{\mu,l} \bullet X + \Phi^{-1}\!\big(\tfrac{1+\beta_l}{2}\big)
 \Big(\sum_{i,j} \sigma_{ij}^2 (w_{R,l})_{ij}^2\Big)^{1/2} \;\;\leq \varepsilon,
\quad l \in \mathcal{N},\\
&R_{\mu,l} \bullet X - \Phi^{-1}\!\big(\tfrac{1+\beta_l}{2}\big)
 \Big(\sum_{i,j} \sigma_{ij}^2 (w_{R,l})_{ij}^2\Big)^{1/2} \;\;\geq -\varepsilon,
\quad l \in \mathcal{N},
\end{align*}
together with covariance consistency \eqref{eq_chance_cov_soc}.
\end{corollary}

\begin{remark}[Chance vs.\ CVaR constraints]
The chance-constrained and CVaR formulations offer complementary perspectives on robustness under payoff uncertainty. 
Chance constraints (Theorem~\ref{theorem_chance}) 
require that the objective and equilibrium conditions hold with high probability $\beta_F,\beta_i$, but do not account for the magnitude of violations in the remaining $(1-\beta)$ fraction of adverse scenarios. 
In contrast, CVaR constraints (Theorem~\ref{theorem_cvar}) explicitly penalize the tail of the distribution by bounding the average shortfall beyond the $\beta$-quantile. Thus, the chance-constrained model is less conservative and may achieve higher 
nominal objectives, while the CVaR model is more robust to extreme events but potentially at the cost of efficiency.
\end{remark}

\begin{proposition}[Feasibility of the design programs]\label{prop:feasibility}
Let $\bar H=\mathbb{E}_\xi[H]$ and $\Sigma=\mathrm{var}(\gamma)$, and write $\circ$ for the Hadamard product. If $\bar H\circ\Sigma$ is invertible, then the feasible regions of the CVaR program of Theorem~\ref{theorem_cvar} and the two-sided chance program of Theorem~\ref{theorem_chance} are nonempty for every tolerance $\varepsilon\ge\varepsilon^\star:=\kappa_\beta\max_i s_i$, where $\kappa_\beta=\Phi^{-1}(\tfrac{1+\beta}{2})$ for the chance design and $\kappa_\beta=\phi(\Phi^{-1}(\beta))/(1-\beta)$ for the CVaR design, $s_i=d_i\big(\sum_j\sigma_{ij}^2\Sigma_{ij}^2 d_j^2\big)^{1/2}$, and $d=(\bar H\circ\Sigma)^{-1}\mathrm{dg}(\Sigma)$ with $\mathrm{dg}(\cdot)$ the vector of diagonal entries. A diagonal linear policy $a=\mathrm{Diag}(d)\,\gamma$ certifies feasibility, and the analogous condition on $\bar H\circ\mathrm{var}(\theta)$ certifies the corrector channel.
\end{proposition}

\begin{proof}{Proof}
The claim is exactly parts (a) and (c) of Theorem~\ref{thm:general_existence} (Appendix~\ref{app:Cal-BCE_existence}), specialized to the action channel. That theorem constructs the diagonal certifying policy $a=\mathrm{Diag}(d)\,\gamma$ with $d=(\bar H\circ\Sigma)^{-1}\mathrm{dg}(\Sigma)$, which is well defined precisely when $\bar H\circ\Sigma$ is invertible; it shows the resulting obedience residual is mean-zero, $\mathbb{E}_\xi[\Lambda^{(X)}_i(H)]=0$, with standard deviation $s_i=d_i(\sum_j\sigma_{ij}^2\Sigma_{ij}^2 d_j^2)^{1/2}$. Substituting this Gaussian residual into the two-sided chance and CVaR reformulations of Theorems~\ref{theorem_cvar} and~\ref{theorem_chance} gives feasibility whenever $\varepsilon\ge\kappa_\beta\max_i s_i=\varepsilon^\star$. The corrector channel is identical with $(\,b,\Psi\,)$ in place of $(\,d,\Sigma\,)$, giving the stated condition on $\bar H\circ\mathrm{var}(\theta)$. \qed
\end{proof}

\begin{theorem}[Joint decentralization of the equilibrium map]\label{thm:decentralization}
Consider the design with Gaussian, independently uncertain coefficients, $\Sigma_H=\mathrm{diag}(\{\sigma_{ij}^2\})$, under either the CVaR constraints of Theorem~\ref{theorem_cvar} or the two-sided chance constraints of Theorem~\ref{theorem_chance}, with tolerance $\varepsilon\ge0$ and confidence $\beta\in(\tfrac12,1)$. Write
\[
m_i:=\mathbb{E}_\xi\big[\Lambda^{(X)}_i(H)\big]=\sum_{j}\bar H_{ij}X_{ij}-X_{i,n+i},
\qquad
s_i:=\Big(\sum_{j}\sigma_{ij}^2\, X_{ij}^2\Big)^{1/2},
\]
and let $\kappa_\beta$ denote the per-constraint multiplier: $\kappa_\beta^{\mathrm{chance}}=\Phi^{-1}\!\big(\tfrac{1+\beta}{2}\big)$ for the two-sided chance design, and $\kappa_\beta^{\mathrm{CVaR}}=\phi\big(\Phi^{-1}(\beta)\big)/(1-\beta)$ for the CVaR design. Then:
\begin{enumerate}
\item[(a)] \emph{(Common decentralization bound; both designs.)} Every feasible action covariance $X$ of \emph{either} formulation satisfies, for each sector $i$,
\begin{equation}\label{eq:decentralization_bound}
\kappa_\beta\, s_i \;\le\; \varepsilon-|m_i| \;\le\; \varepsilon,
\qquad\text{hence}\qquad
|X_{ij}|\;\le\;\frac{s_i}{\sigma_{ij}}\;\le\;\frac{\varepsilon}{\kappa_\beta\,\sigma_{ij}}\quad\text{for every }j\text{ with }\sigma_{ij}>0,
\end{equation}
with $\kappa_\beta=\kappa_\beta^{\mathrm{chance}}$ or $\kappa_\beta^{\mathrm{CVaR}}$ as appropriate. In particular, for each cross-agent pair $i\neq j$ with $\sigma_{ij}>0$, $|X_{ij}|\to 0$ as $\beta\to1$ ($\kappa_\beta\to\infty$) or $\varepsilon\to0$: the block $X_{aa}=\mathrm{var}(a)$ becomes diagonal and equilibrium actions decouple across agents. The cap scales as $1/\sigma_{ij}$, so entries tied to highly uncertain coefficients are suppressed first; when cross-impacts are less precisely identified than own-impacts ($\sigma_{ij}\ge\sigma_{ii}$, $j\neq i$), $X_{aa}$ is driven to be diagonally dominant.
\item[(b)] \emph{(Matched-tolerance ordering.)} Suppose $\beta\in[\tfrac12,1)$. Then
\[
\kappa_\beta^{\mathrm{CVaR}}\ \ge\ \kappa_\beta^{\mathrm{chance}}\ \ge\ \Phi^{-1}(\beta),
\]
so at a common tolerance $\varepsilon$ and confidence $\beta$ the CVaR design caps every off-diagonal covariance $X_{ij}$ ($i\neq j$) at least as tightly as the two-sided chance design, which is in turn tighter than the unconstrained baseline.

\item[(c)] \emph{(Realized ordering via the feasibility threshold.)} The two formulations need not---and in general do not---operate at the same $\varepsilon$. Each is feasible only above its own threshold $\varepsilon^\star_{(\cdot)}=\kappa_\beta^{(\cdot)}\max_i s_i^{\star}$ (Proposition~\ref{prop:feasibility}), where $s_i^{\star}$ is the residual standard deviation at the certifying policy. Writing $\varepsilon_{\mathrm{ch}},\varepsilon_{\mathrm{cv}}$ for the tolerances actually used, the realized caps compare as
\[
\frac{\varepsilon_{\mathrm{cv}}}{\kappa_\beta^{\mathrm{CVaR}}\,\sigma_{ij}}
\;\;\gtrless\;\;
\frac{\varepsilon_{\mathrm{ch}}}{\kappa_\beta^{\mathrm{chance}}\,\sigma_{ij}}
\quad\Longleftrightarrow\quad
\frac{\varepsilon_{\mathrm{cv}}}{\varepsilon_{\mathrm{ch}}}\;\;\gtrless\;\;\frac{\kappa_\beta^{\mathrm{CVaR}}}{\kappa_\beta^{\mathrm{chance}}}.
\]
Hence the matched-tolerance advantage of CVaR in (b) is realized only when the CVaR design is run at a tolerance no larger than $\big(\kappa_\beta^{\mathrm{CVaR}}/\kappa_\beta^{\mathrm{chance}}\big)\,\varepsilon_{\mathrm{ch}}$. Because the larger CVaR multiplier raises its feasibility floor $\varepsilon^\star_{\mathrm{CVaR}}$, a design operating at its own threshold may instead satisfy $\varepsilon_{\mathrm{cv}}/\varepsilon_{\mathrm{ch}}>\kappa_\beta^{\mathrm{CVaR}}/\kappa_\beta^{\mathrm{chance}}$, in which case the probabilistic design attains the tighter realized cap. The empirical ordering is therefore governed by the threshold ratio $\varepsilon^\star_{\mathrm{CVaR}}/\varepsilon^\star_{\mathrm{chance}}$ and is calibration-dependent.
\end{enumerate}
\end{theorem}

\begin{proof}{Proof}
Under $\Sigma_H=\mathrm{diag}(\{\sigma_{ij}^2\})$ and Gaussian $H$, the affine functional $\Lambda^{(X)}_i(H)=\sum_j H_{ij}X_{ij}-X_{i,n+i}$ is Gaussian with mean $m_i$ and variance $\sum_j \sigma_{ij}^2 X_{ij}^2=s_i^2$ (the weighted-$\ell_2$ form of Corollary~\ref{cor:independent} applied to $R_l\bullet X=\Lambda^{(X)}_l(H)$).

\emph{(a)} \emph{Chance design.} The two halves of the band require $Pr_\xi(\Lambda^{(X)}_i\le\varepsilon)\ge\tfrac{1+\beta}{2}$ and $Pr_\xi(\Lambda^{(X)}_i\ge-\varepsilon)\ge\tfrac{1+\beta}{2}$, i.e.\ $m_i+\kappa_\beta^{\mathrm{chance}} s_i\le\varepsilon$ and $-m_i+\kappa_\beta^{\mathrm{chance}} s_i\le\varepsilon$. Each implies $\kappa_\beta^{\mathrm{chance}} s_i\le\varepsilon-|m_i|$. \emph{CVaR design.} For $Z\sim\mathcal N(m,s^2)$ one has $\mathrm{CVaR}_\beta(Z)=m+\kappa_\beta^{\mathrm{CVaR}} s$; the two-sided constraints $\mathrm{CVaR}_\beta(\pm\Lambda^{(X)}_i-\varepsilon)\le0$ read $m_i+\kappa_\beta^{\mathrm{CVaR}} s_i\le\varepsilon$ and $-m_i+\kappa_\beta^{\mathrm{CVaR}} s_i\le\varepsilon$, giving $\kappa_\beta^{\mathrm{CVaR}} s_i\le\varepsilon-|m_i|$. In both cases $\kappa_\beta s_i\le\varepsilon-|m_i|\le\varepsilon$, which is \eqref{eq:decentralization_bound}; the per-entry bound follows from $\sigma_{ij}|X_{ij}|\le s_i$, and the limit and selective-suppression statements are immediate.

\emph{(b)} For the right inequality, $\tfrac{1+\beta}{2}\ge\beta$ whenever $\beta\le1$, and $\Phi^{-1}$ is increasing, so $\kappa^{\mathrm{chance}}_\beta=\Phi^{-1}(\tfrac{1+\beta}{2})\ge\Phi^{-1}(\beta)$. For the left inequality, use the expected-shortfall representation of the standardized CVaR multiplier,
\[
\kappa^{\mathrm{CVaR}}_\beta=\frac{\phi(\Phi^{-1}(\beta))}{1-\beta}=\frac{1}{1-\beta}\int_\beta^1\Phi^{-1}(u)\,du,
\]
where the second equality follows from $\int x\phi(x)\,dx=-\phi(x)$ under $u=\Phi(x)$; thus $\kappa^{\mathrm{CVaR}}_\beta$ is the mean of the quantile function $\Phi^{-1}$ over $[\beta,1]$. Differentiating $\Phi^{-1}$ twice gives $(\Phi^{-1})''(u)=\Phi^{-1}(u)/\phi(\Phi^{-1}(u))^2$, which is nonnegative for $u\ge\tfrac12$; hence $\Phi^{-1}$ is convex on $[\beta,1]\subseteq[\tfrac12,1)$. By the Hermite--Hadamard inequality, the mean of a convex function over an interval is at least its value at the midpoint $\tfrac{\beta+1}{2}$:
\[
\kappa^{\mathrm{CVaR}}_\beta=\frac{1}{1-\beta}\int_\beta^1\Phi^{-1}(u)\,du\ \ge\ \Phi^{-1}\!\Big(\tfrac{\beta+1}{2}\Big)=\kappa^{\mathrm{chance}}_\beta .
\]
Combining, $\kappa^{\mathrm{CVaR}}_\beta\ge\kappa^{\mathrm{chance}}_\beta\ge\Phi^{-1}(\beta)$ for all $\beta\in[\tfrac12,1)$. At common $(\varepsilon,\beta)$ the per-entry caps $\varepsilon/(\kappa_\beta\sigma_{ij})$ are ordered inversely to the multipliers. 

\emph{(c)} The feasibility thresholds $\varepsilon^\star_{(\cdot)}=\kappa_\beta^{(\cdot)}\max_i s_i^\star$ follow from Proposition~\ref{prop:feasibility}. Dividing the two realized caps $\varepsilon_{\mathrm{cv}}/(\kappa_\beta^{\mathrm{CVaR}}\sigma_{ij})$ and $\varepsilon_{\mathrm{ch}}/(\kappa_\beta^{\mathrm{chance}}\sigma_{ij})$ gives the stated equivalence, and substituting $\varepsilon_{(\cdot)}=\varepsilon^\star_{(\cdot)}$ shows the threshold-ratio dependence. \qed
\end{proof}

\section{Numerical Experiments}
\label{sec:experiments}

This section presents numerical experiments evaluating the tractability of the convex reformulations developed in Sections~\ref{sec_cvar}-\ref{sec_chance} and to compare welfare and equilibrium properties across formulations.  

\subsection{Experimental Setup}
The investment game from Example~\ref{ex:investment} is implemented, where each agent corresponds to a sector exchange traded fund (ETF).  
Daily price and volume data for U.S. sector ETFs (2015-2024) are used to estimate the payoff matrix $H$, which encodes both self-costs and cross-impacts between sectors.  

The matrix $H$ is estimated using ridge regression, which stabilizes coefficient estimates in the presence of strong collinearity across sector returns.  
This shrinkage is particularly important in financial panel data, where sector exposures are highly correlated and naive least-squares estimates can be unstable.  
To capture statistical variability, a block bootstrap procedure resamples returns in blocks of trading days, thereby preserving temporal dependence structures such as volatility clustering.  
The empirical covariance of returns provides an estimate of $\mathrm{var}(\gamma)$, while bootstrap variability in regression coefficients yields a distributional model $\xi$ for $H$.

Consistent with the theory of Section~\ref{sec:lqg}, the uncertain primitive is the payoff matrix $H\sim\xi$, and the per-agent uncertainty state is $\theta=\mathrm{diag}(H)$, whose covariance $\mathrm{var}(\theta)$ is read off the bootstrap covariance of the own-coefficients $\{H_{ii}\}$. The action-channel design variable $X$ is obtained by solving the CVaR and chance-constrained programs of Theorems~\ref{theorem_cvar} and~\ref{theorem_chance}, in which the \emph{full} (including off-diagonal) uncertainty in $H$ enters the robust equilibrium constraints; the off-diagonal coefficients are therefore never treated as known. The corrector covariance $G$ is not optimized separately: by Theorem~\ref{thm:general_existence}(d), the diagonal calibration $b_i=1/H_{ii}$ satisfies the corrector-channel constraint \eqref{eq_affected_inequality_app_cvar_2} of Lemma~\ref{lem:linear_cvar_X} exactly for every realization of $H$, so equilibrium consistency on the corrector channel holds by construction once $X$ is designed.

From these estimates, $J=1500$ scenarios of $(H,F)$ are generated, and both chance-constrained and CVaR-based formulations are solved.  
Unless otherwise specified, parameters are set to $\beta_F=0.95$ and $\beta_i=0.95$.
All semidefinite programs are implemented in \texttt{CVXPY} and solved using MOSEK.  

\subsection{Models Compared}
Results are reported for three cases:
\begin{enumerate}
    \item \textbf{Baseline (K0):} welfare under the empirical distribution without intervention.  
    \item \textbf{Probabilistic information design:} constraints enforced with 95\% probability under $\xi$.
    \item \textbf{CVaR information design:} constraints enforced in a tail-risk sense using 95\% CVaR.
\end{enumerate}

Each formulation produces an optimal action covariance $X^\star$ and an equilibrium map $K$. Concretely, from $X^\star$ one recover $K$ as the best linear predictor of equilibrium actions from payoff states, $K=\mathrm{cov}(a,\gamma)\,\mathrm{var}(\gamma)^{-1}$, evaluated on the corresponding blocks of $X^\star$.  

\subsection{Welfare Outcomes}
Table~\ref{tab:welfare_summary} summarizes normalized welfare across $4{,}000$ evaluation scenarios, with welfare divided by $\mathrm{tr}(K\Sigma K^\top)$ to remove action-scale dependence. All three designs achieve positive mean welfare, but they differ sharply in the tail. The baseline retains a negative CVaR@95\% ($-6.9\times10^{-6}$), exposing residual downside in the worst $5\%$ of payoff draws. The probabilistic design eliminates negative realizations entirely ($P(w>0)=1$) and raises mean welfare roughly $26\times$, though with correspondingly higher variance. The CVaR design occupies the middle ground: about a $1.6\times$ gain in mean welfare over the baseline, together with a strictly positive CVaR@95\% ($9.3\times10^{-6}$) and $P(w>0)=99.3\%$---robust tail protection achieved by penalizing worst-case welfare losses directly rather than by suppressing activity.

\begin{table}[h]
\centering
\caption{Summary statistics of normalized welfare distributions across $4{,}000$ evaluation scenarios. Welfare is divided by $\mathrm{tr}(K\Sigma K^\top)$ to remove action-scale dependence. CVaR@95\% reports the expected welfare in the worst 5\% of payoff draws. Parameters: $\beta_F=\beta_i=0.95$, $J=1500$ optimization scenarios, MOSEK solver.}
\label{tab:welfare_summary}
\small
\begin{tabular}{lrrrrrr}
\toprule
Model & Mean & Std & Median & 5\% Quantile & CVaR@95\% & $P(\text{welfare}>0)$ \\
\midrule
Baseline ($K_0$)  & $8.18 \times 10^{-5}$ & $9.53 \times 10^{-5}$ & $5.09 \times 10^{-5}$ & $8.30 \times 10^{-6}$ & $-6.86 \times 10^{-6}$ & $0.983$ \\
Probabilistic     & $2.12 \times 10^{-3}$ & $1.88 \times 10^{-3}$ & $1.54 \times 10^{-3}$ & $4.45 \times 10^{-4}$ & $3.44 \times 10^{-4}$  & $1.000$ \\
CVaR              & $1.28 \times 10^{-4}$ & $1.21 \times 10^{-4}$ & $9.03 \times 10^{-5}$ & $1.94 \times 10^{-5}$ & $9.28 \times 10^{-6}$  & $0.993$ \\
\bottomrule
\end{tabular}
\end{table}

\begin{figure}[t]
\centering
\includegraphics[width=0.85\linewidth]{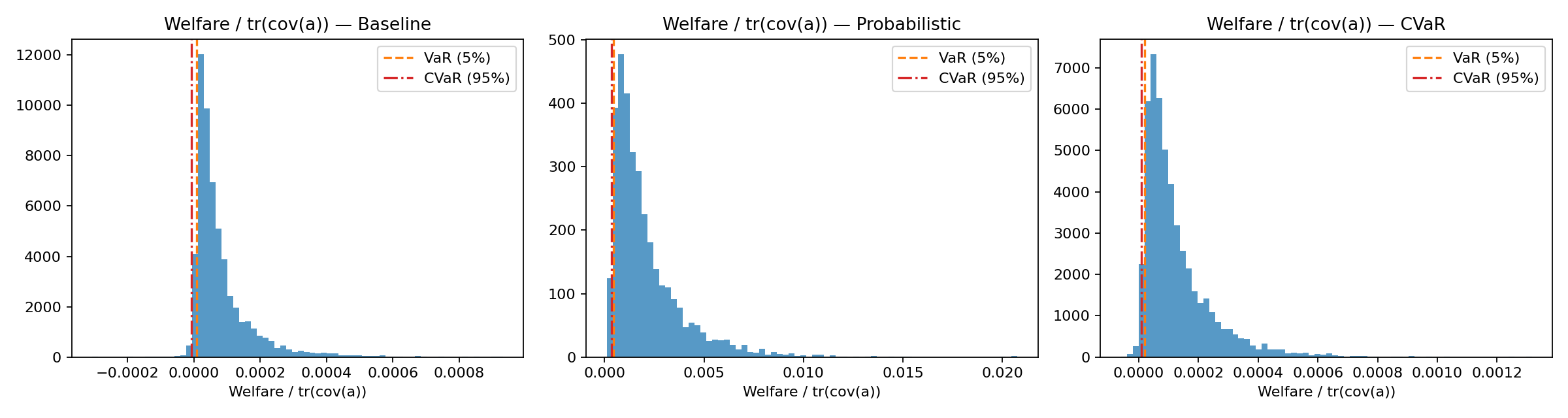}
\caption{Normalized welfare distributions across $4{,}000$ evaluation scenarios under three
formulations ($\beta=0.95$).
Dashed orange lines mark the 5\% VaR; red dash-dotted lines mark the CVaR@95\%.
\emph{Left:} Baseline ($K_0$), which achieves positive mean welfare ($8.2\times10^{-5}$) but
retains a small negative left tail (CVaR@95\%$=-6.9\times10^{-6}$), reflecting residual
downside exposure from unmanaged cross-sector spillovers.
\emph{Middle:} Probabilistic design, which raises mean welfare by approximately $26\times$
over the baseline ($2.1\times10^{-3}$) with all realizations positive ($P(w>0)=1.0$),
exploiting cross-sector co-movements at the cost of high variance.
\emph{Right:} CVaR design, which achieves an intermediate mean ($1.3\times10^{-4}$),
strictly positive CVaR@95\% ($9.3\times10^{-6}$), and $P(w>0)=99.3\%$ through explicit
penalization of tail-welfare losses across optimization scenarios.}
\label{fig:welfare_histograms}
\end{figure}

\begin{figure}[t]
\centering
\includegraphics[width=0.4\linewidth]{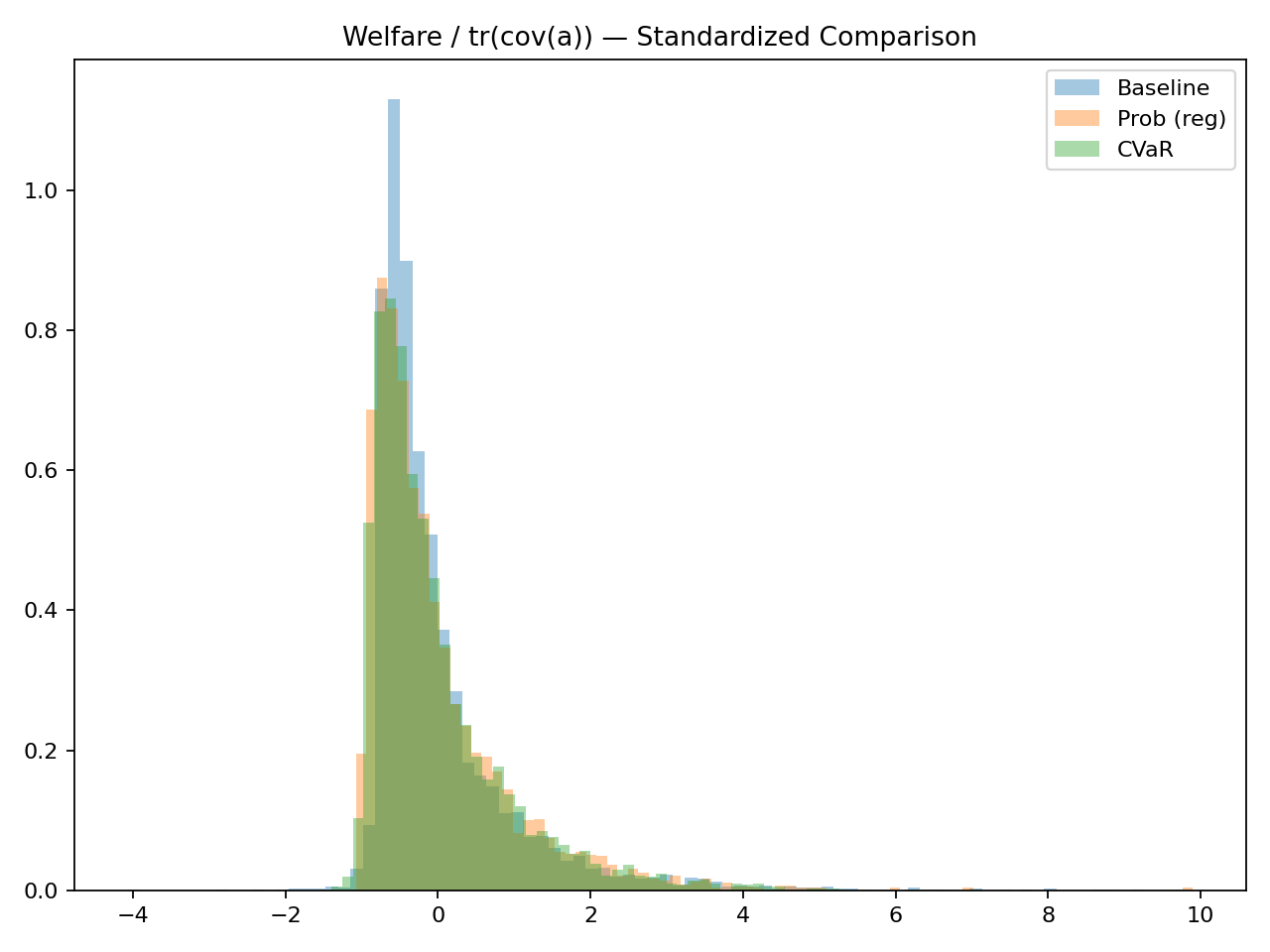}
\caption{Standardized welfare distributions (rescaled to mean zero, unit variance) across the
three formulations.
Once standardized, baseline (blue) and CVaR (green) largely overlap, both exhibiting a
moderate left tail and a concentrated body.
The probabilistic design (orange) stands apart: it has the same left-tail shape but a
substantially longer right tail, reflecting the amplified upside from exploiting cross-sector
co-movements (raw std $\approx1.9\times10^{-3}$ vs.\ $9.5\times10^{-5}$ for baseline and
$1.2\times10^{-4}$ for CVaR).
This confirms that the welfare improvement of the probabilistic design is primarily driven by
capturing high-welfare tail events, rather than by compressing the distribution.}
\label{fig:welfare_standardized}
\end{figure}

Because all welfare values are normalized by the equilibrium action variance $\mathrm{tr}(K\Sigma K^\top)$, these comparisons isolate welfare \emph{per unit of action} and are not driven by differences in action scale across formulations. On this common footing, the contrast is clear: the probabilistic design wins on mean welfare by exploiting cross-sector co-movements, while the CVaR design wins on tail protection, and the baseline is dominated on both.

\subsection{Equilibrium Maps}
These structural differences are the empirical counterpart of the joint decentralization theorem (Theorem~\ref{thm:decentralization}), which bounds not $K$ itself but the resulting \emph{action covariance} $X_{aa}=K\Sigma K^\top$: every feasible design satisfies $|X_{ij}|\le s_i/\sigma_{ij}$ from the decentralization bound~\eqref{eq:decentralization_bound}. At the level of $X_{aa}$, both designs decentralize relative to the baseline---the off-diagonal absolute maximum falls from $5.86\times10^4$ (baseline) to $1.65\times10^4$ (CVaR) and to $328$ (probabilistic)---consistent with part~(a) of the theorem. Part~(b) predicts that, \emph{at equal $(\varepsilon,\beta)$}, the CVaR design caps off-diagonal covariances at least as tightly as the chance design, since $\kappa_\beta^{\mathrm{CVaR}}\ge\kappa_\beta^{\mathrm{chance}}$. The realized ordering observed here is reversed because the two formulations are calibrated at their own design-specific feasibility thresholds rather than a common one (part~(c)): in this calibration the probabilistic design ends up tighter, leaving the CVaR design with a richer $K$ and $X_{aa}$ structure that it uses for direct tail-welfare protection rather than for decoupling. The limiting behavior matches the theory exactly: when sector shocks are independent ($\Sigma=\mathrm{var}(\gamma)$ diagonal), the map $K=\mathrm{cov}(a,\gamma)\,\Sigma^{-1}$ becomes diagonal as $\varepsilon\to0$---each sector's equilibrium action tracks only its own fundamental---recovering the closed-form diagonal construction of Theorem~\ref{thm:general_existence}(d). Economically, the probabilistic design achieves robustness by collapsing the equilibrium map toward a diagonal---decoupling sectors and compressing own-impacts alike---whereas the CVaR design instead protects the welfare tail directly, tolerating the cross-sector channels that carry the most tail-relevant information. The estimation-robustness experiment (Section~\ref{sec:estimation_robustness}) confirms these structures are not artifacts: the ridge estimator produces richer off-diagonal structure than OLS across both designs.

\begin{figure}[t]
\centering
\includegraphics[width=0.85\linewidth]{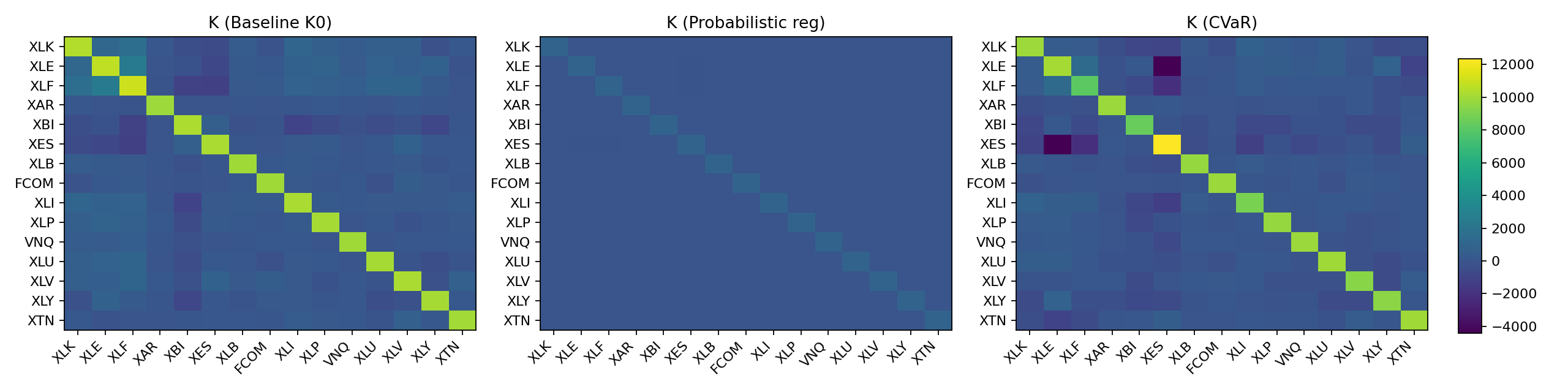}
\caption{Heatmaps of the estimated equilibrium map $K$ under three formulations.
\emph{Left:} Baseline ($K_0$), with own-impacts clustered near $10{,}000$ (range $9{,}913$--$11{,}062$) and off-diagonal entries spanning $-1{,}220$ to $2{,}511$.
\emph{Middle:} Probabilistic design, which compresses own-impacts to a tight band near $900$ (range $897$--$917$) and suppresses off-diagonal entries to within $\pm57$, indicating that chance constraints strongly regularize cross-sector interactions under this estimation.
\emph{Right:} CVaR design, with own-impacts dispersed over $8{,}154$--$12{,}348$ and off-diagonal entries spanning $-4{,}416$ to $1{,}362$---the widest off-diagonal range of the three---reflecting that CVaR targets tail welfare directly and retains substantial cross-sector channels rather than suppressing them.
This illustrates that the probabilistic design achieves robustness through sector decoupling, while the CVaR design achieves tail-robust welfare while preserving baseline-level cross-sector interaction.}
\label{fig:K_maps}
\end{figure}

\begin{figure}[t]
\centering
\includegraphics[width=0.6\linewidth]{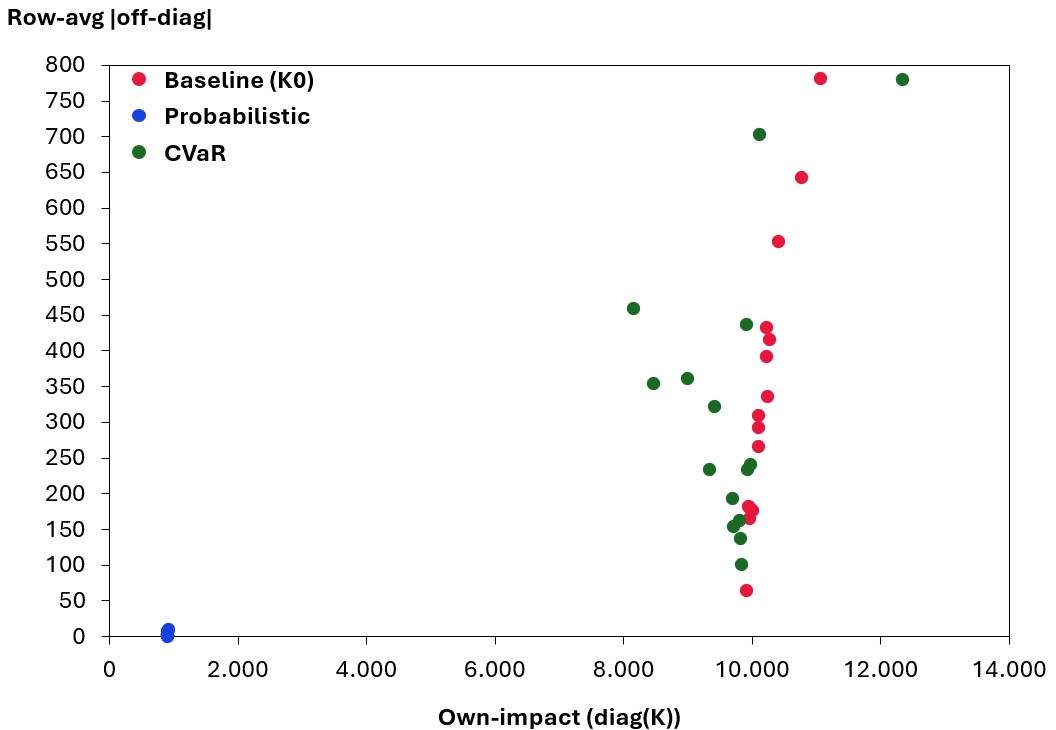}
\caption{Equilibrium structure across the three formulations, shown as a single overlaid scatter. Each point is one sector ETF, plotting its own-impact (diagonal entry of $K$) against its row-average absolute off-diagonal impact (cross-sector spillover); the formulations are distinguished by color. The baseline ($K_0$, red) clusters at own-impacts of $9{,}913$--$11{,}062$ with moderate spillovers ($65$--$781$). The probabilistic design (blue) collapses to a tight cluster at own-impacts near $900$ (range $897$--$917$) with near-zero spillovers ($0.8$--$10.4$)---the most structurally concentrated and decoupled outcome, compressing both own- and cross-sector terms by roughly an order of magnitude relative to the baseline. The CVaR design (green) keeps own-impacts at baseline scale but more dispersed ($8{,}154$--$12{,}348$) and retains spillovers comparable to the baseline ($101$--$780$), substantially less compressed than the probabilistic design. Together with Figure~\ref{fig:K_maps}, this shows that in this calibration the probabilistic design achieves robustness primarily through sector decoupling, whereas the CVaR design achieves tail-robust welfare while preserving baseline-level cross-sector interaction.}
\label{fig:K_scatter_all}
\end{figure}

\subsection{Sensitivity to Confidence Level}
\label{sec:beta_sweep}

Table~\ref{tab:beta_sweep} reports normalized welfare statistics as $\beta$ varies over $\{0.90,\,0.95,\,0.99\}$ for both formulations (Figure~\ref{fig:beta_sweep}).

\begin{table}[h]
\centering
\caption{Normalized welfare statistics under varying confidence levels $\beta$ for CVaR and probabilistic formulations ($J=200$ scenarios, quick mode, MOSEK solver). All welfare entries should be read at the scale $\times10^{-3}$.}
\label{tab:beta_sweep}
\small
\begin{tabular}{llrrrr}
\toprule
$\beta$ & Model & Mean ($\times10^{-3}$) & Std ($\times10^{-3}$) & CVaR@95\% ($\times10^{-3}$) & $P(w>0)$ \\
\midrule
$0.90$ & CVaR  & $0.121$ & $0.120$ & $0.0032$ & $0.986$ \\
$0.90$ & Prob  & $2.145$ & $1.894$ & $0.339$  & $1.000$ \\
$0.95$ & CVaR  & $0.125$ & $0.129$ & $0.0045$ & $0.991$ \\
$0.95$ & Prob  & $2.133$ & $1.934$ & $0.349$  & $1.000$ \\
$0.99$ & CVaR  & $0.123$ & $0.119$ & $0.0095$ & $0.996$ \\
$0.99$ & Prob  & $2.155$ & $1.992$ & $0.342$  & $1.000$ \\
\bottomrule
\end{tabular}
\end{table}

\begin{figure}[t]
\centering
\includegraphics[width=0.75\linewidth]{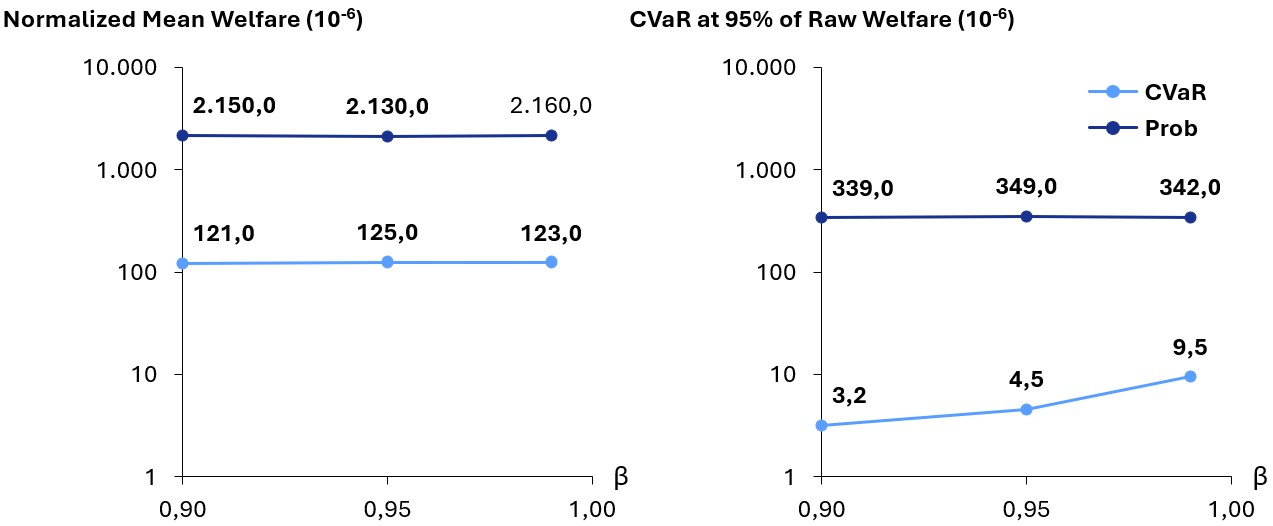}
\caption{Mean normalized welfare (\emph{left}) and CVaR@95\% of normalized welfare (\emph{right}) as a function of confidence level $\beta\in\{0.90,0.95,0.99\}$. The CVaR design's tail metric improves monotonically with $\beta$ (from $3.2\times10^{-6}$ at $\beta=0.90$ to $9.5\times10^{-6}$ at $\beta=0.99$), while its mean welfare remains nearly flat (about a $3\%$ range), confirming that tighter tail constraints improve worst-case outcomes at minimal efficiency cost. The probabilistic design's welfare is insensitive to $\beta$ in this range.}
\label{fig:beta_sweep}
\end{figure}

The CVaR@95\% of the CVaR design improves monotonically with $\beta$, from $3.2\times10^{-6}$ at $\beta=0.90$ to $9.5\times10^{-6}$ at $\beta=0.99$, confirming that a higher confidence level enforces deeper tail protection, consistent with the role of $\beta$ in the CVaR formulation (Theorem~\ref{theorem_cvar}). Its probability of positive welfare rises in step ($0.986\to0.991\to0.996$). Mean welfare varies by only about $3\%$ across all values of $\beta$, so the tail-protection gain is achieved at negligible efficiency cost. The probabilistic design's welfare is nearly flat across $\beta$, indicating robustness of the chance-constrained solution to the confidence level in this range.

\subsection{Alternative Design Objectives}
\label{sec:obj_sweep}

Table~\ref{tab:obj_sweep} reports raw (unnormalized) welfare evaluated on a common scenario set across two objectives: welfare maximization (Example~\ref{ex_social_welfare}) and contagion suppression (Example~\ref{ex_contagion}, maximizing $-\sum_{i\neq j}(X_{aa})_{ij}$). Using a single raw welfare criterion makes all rows directly comparable. Figure~\ref{fig:obj_sweep} plots welfare and CVaR@95\% across objectives.

\begin{table}[h]
\centering
\caption{Raw (unnormalized) welfare evaluated on a common set of $4{,}000$ scenarios for all designs, regardless of their optimization objective. The contagion suppression objective (Example~\ref{ex_contagion}) maximizes $-\sum_{i\neq j}(X_{aa})_{ij}$. Baseline $K_0=\bar{H}^{-1}$ is included as a reference.}
\label{tab:obj_sweep}
\small
\begin{tabular}{llrrrr}
\toprule
Objective & Model & Mean welfare & CVaR@95\% & $P(w>0)$ & $\mathrm{tr}(K\Sigma K^\top)$ \\
\midrule
Baseline             & $K_0$ & $42.8$              & $-3.83$             & $0.982$ & $5.07\times10^{5}$ \\
Welfare              & CVaR  & $37.9$              & $3.37$              & $0.995$ & $3.03\times10^{5}$ \\
Welfare              & Prob  & $6.21$              & $1.03$              & $1.000$ & $2.86\times10^{3}$ \\
Contagion supp.      & CVaR  & $10.6$              & $-15.0$             & $0.803$ & $3.18\times10^{5}$ \\
Contagion supp.      & Prob  & $2.0\times10^{-3}$  & $4.9\times10^{-4}$  & $1.000$ & $6.1\times10^{-4}$ \\
\bottomrule
\end{tabular}
\end{table}

\begin{figure}[t]
\centering
\includegraphics[width=0.75\linewidth]{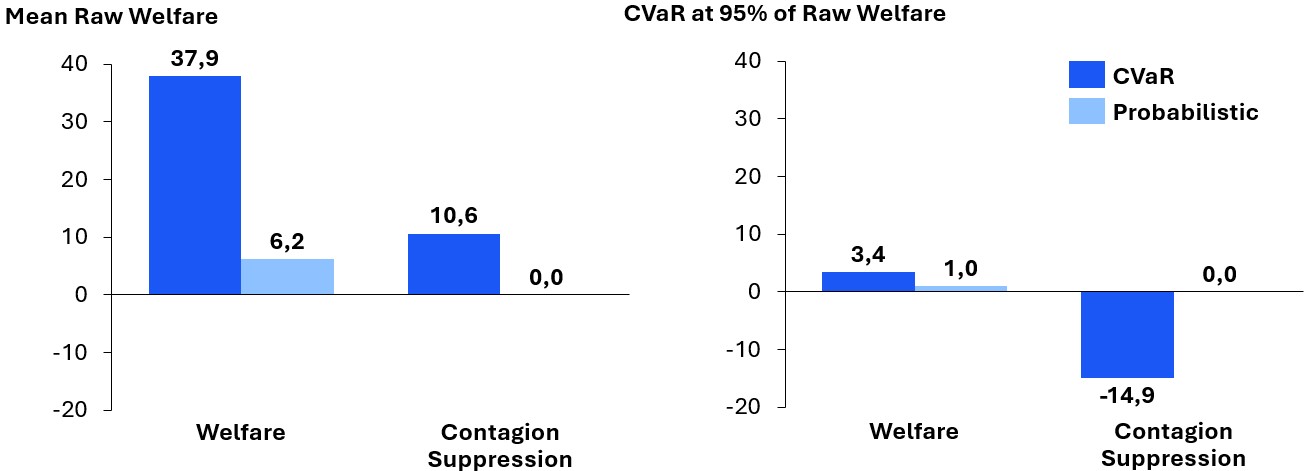}
\caption{Raw welfare and CVaR@95\% across design objectives (welfare maximization and contagion suppression), evaluated on a common $4{,}000$-scenario set; the baseline $K_0$ is reported in Table~\ref{tab:obj_sweep} for reference. Under the welfare objective, the CVaR design converts the baseline's negative welfare tail ($-3.83$) into the highest positive CVaR@95\% ($3.37$) while retaining most of the baseline mean welfare. Under contagion suppression, the probabilistic design collapses to $K\approx0$ (welfare and tail $\approx0$, transmitting no information), whereas the CVaR design retains a non-degenerate equilibrium with positive mean welfare ($10.6$)---but a sharply negative welfare tail (CVaR@95\% $=-15.0$, $P(w>0)=0.80$), since welfare is no longer the optimization target. CVaR's per-scenario constraints thus prevent the degenerate collapse of the probabilistic formulation under non-welfare objectives, without by itself protecting welfare.}
\label{fig:obj_sweep}
\end{figure}

Two findings emerge from Table~\ref{tab:obj_sweep}. First, the welfare-CVaR design preserves a large share of the baseline's mean welfare ($37.9$ vs.\ $42.8$) while converting the welfare CVaR@95\% from negative ($-3.83$) to positive ($3.37$), at modest efficiency cost; welfare-Prob uses a much smaller action scale ($\mathrm{tr}(K\Sigma K^\top)=2{,}860$ vs.\ $303{,}000$), yielding lower raw mean welfare ($6.21$) but guaranteed positive outcomes ($P(w>0)=1$). Second, and more strikingly, the contagion suppression objective exposes a structural difference between the formulations. The probabilistic design collapses to $K\approx0$ ($\mathrm{tr}\approx6\times10^{-4}$, mean welfare $\approx2\times10^{-3}$): it transmits no information. The CVaR design instead retains a large action scale ($\mathrm{tr}\approx3.2\times10^5$) and a non-degenerate equilibrium with positive mean welfare ($10.6$)---but, because the optimization target is now contagion rather than welfare, its welfare tail turns sharply negative (CVaR@95\% $=-15.0$, $P(w>0)=0.80$), worse even than the baseline ($-3.83$). The probabilistic collapse is structural: for a constant-$F$ objective ($\mathrm{std}_F=0$), the probabilistic chance constraint reduces to a single linear inequality on $X_{aa}$, which the optimizer satisfies by inflating anti-correlated action noise while keeping $K\approx0$---no information is transmitted and welfare vanishes. The CVaR formulation's per-scenario equilibrium constraints couple $X_{ag}$ to each $H^{(j)}$ realization individually, preventing this degenerate solution and producing a non-trivial policy. The lesson is not that CVaR yields better welfare under a non-welfare objective---it does not, as its negative tail shows---but that its scenario-level structure enforces economically meaningful, non-degenerate equilibria even when the optimization target diverges from welfare, whereas the probabilistic formulation degenerates.

\subsection{Stress-Test Analysis: Technology Sector Shock}
\label{sec:stress}

Table~\ref{tab:stress} evaluates normalized welfare under a severe technology-sector shock: the XLK sector return is shifted down by $5\sigma$ relative to its unconditional distribution, simulating a sudden market dislocation. Figure~\ref{fig:stress} shows the distribution shift.

\begin{table}[h]
\centering
\caption{Normalized welfare under normal vs.\ stressed (XLK $-5\sigma$) scenarios for three designs ($J=4{,}000$ evaluation scenarios, MOSEK solver).}
\small
\label{tab:stress}
\begin{tabular}{llrrr}
\toprule
Design        & Scenario & Mean & CVaR@95\% & $P(w>0)$ \\
\midrule
Baseline      & Normal              & $8.31\times10^{-5}$  & $-1.06\times10^{-5}$ & $0.976$ \\
Baseline      & Stress (XLK $-5\sigma$) & $2.00\times10^{-4}$  & $1.99\times10^{-5}$  & $0.991$ \\
Probabilistic & Normal              & $2.18\times10^{-3}$  & $3.57\times10^{-4}$  & $1.000$ \\
Probabilistic & Stress (XLK $-5\sigma$) & $5.58\times10^{-3}$  & $2.77\times10^{-3}$  & $1.000$ \\
CVaR          & Normal              & $1.31\times10^{-4}$  & $9.04\times10^{-6}$  & $0.994$ \\
CVaR          & Stress (XLK $-5\sigma$) & $3.38\times10^{-4}$  & $7.45\times10^{-5}$  & $0.997$ \\
\bottomrule
\end{tabular}
\end{table}

\begin{figure}[t]
\centering
\includegraphics[width=0.85\linewidth]{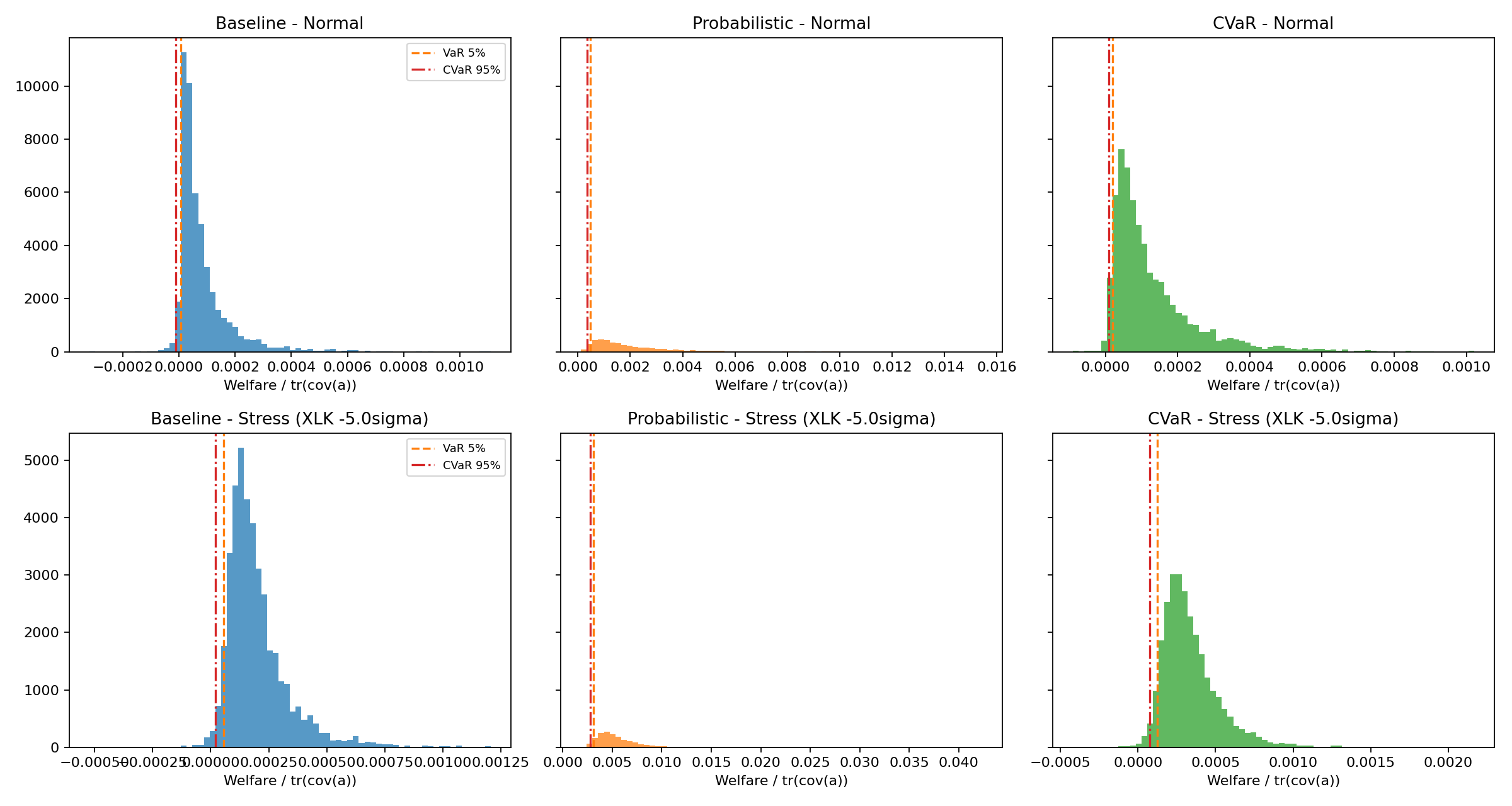}
\caption{Welfare distributions under normal (\emph{top row}) and stressed
(\emph{bottom row}, XLK $-5\sigma$) conditions for the baseline, probabilistic,
and CVaR designs.
Under stress, mean welfare rises for all three designs: the large negative
XLK payoff shock shifts the quadratic objective in a direction that benefits
agents whose equilibrium maps route the XLK signal into positive net payoffs.
The CVaR design exhibits a contained stress response (mean welfare gap
$+2.1\times10^{-4}$) relative to the probabilistic design (gap $+3.4\times10^{-3}$),
reflecting that its tail-risk penalization directly limits how much the equilibrium
map amplifies extreme $\gamma$ realizations---including sector-level shocks of this
type.}
\label{fig:stress}
\end{figure}

Under the XLK stress shock, all three designs display an increase in mean welfare. This is not paradoxical: the welfare metric is utilitarian, summing payoffs across all sectors, so a large loss concentrated in one sector can be more than offset by gains elsewhere. The mechanism is the redistributive geometry of the LQG objective: a large negative payoff shock for one sector reduces that sector's equilibrium action, which---through the cross-sector off-diagonals of $H$---can raise net welfare for the others. What distinguishes the designs is how much they \emph{amplify} the shock. The CVaR design's mean welfare gap (stressed minus normal: $+2.1\times10^{-4}$) is more than an order of magnitude smaller than the probabilistic design's ($+3.4\times10^{-3}$), and its CVaR@95\% gap ($+6.5\times10^{-5}$) is similarly contained against the probabilistic design's $+2.4\times10^{-3}$. The containment is a direct effect of CVaR's tail penalization rather than of structural decoupling: although the probabilistic design produces the more diagonal equilibrium map in this calibration, it is the CVaR design that reacts least to the dislocation, because its per-scenario constraints bound worst-case welfare across the very payoff draws---including extreme sector-level shocks---that the stress scenario probes. The CVaR design is thus the least reactive to sector-specific dislocations, consistent with the role of tail-risk control rather than off-diagonal suppression.

\subsection{Estimation Robustness}
\label{sec:estimation_robustness}
Table~\ref{tab:estimation_robustness} reports CVaR design welfare and equilibrium-map statistics as the estimation method (ridge regression vs.\ OLS) and bootstrap block length (10, 25, 50 trading days) vary. Figure~\ref{fig:estimation_robustness} displays the same welfare metrics and the $K$-structure scatter across these configurations.
\begin{table}[h]
\centering
\caption{CVaR information design welfare and equilibrium map statistics under varying estimation methods and block lengths ($J=200$ quick-mode scenarios, MOSEK solver). $K$ off-diag mean reports the mean absolute off-diagonal entry of the $15\times15$ equilibrium map.}
\label{tab:estimation_robustness}
\small
\begin{tabular}{llrrrr}
\toprule
Estimator & Block len. & Mean welfare & CVaR@95\% & $P(w>0)$ & $K$ off-diag mean \\
\midrule
Ridge & $10$ & $1.38\times10^{-4}$ & $1.35\times10^{-5}$ & $0.997$ & $445$ \\
Ridge & $25$ & $1.32\times10^{-4}$ & $1.21\times10^{-5}$ & $0.997$ & $443$ \\
Ridge & $50$ & $1.29\times10^{-4}$ & $1.28\times10^{-5}$ & $0.998$ & $367$ \\
OLS   & $10$ & $1.01\times10^{-4}$ & $1.62\times10^{-5}$ & $1.000$ &  $47$ \\
OLS   & $25$ & $9.85\times10^{-5}$ & $1.55\times10^{-5}$ & $1.000$ &  $47$ \\
OLS   & $50$ & $1.03\times10^{-4}$ & $1.64\times10^{-5}$ & $1.000$ &  $47$ \\
\bottomrule
\end{tabular}
\end{table}
\begin{figure}[t]
\centering
\begin{minipage}[b]{0.62\linewidth}
\centering
\includegraphics[width=\linewidth]{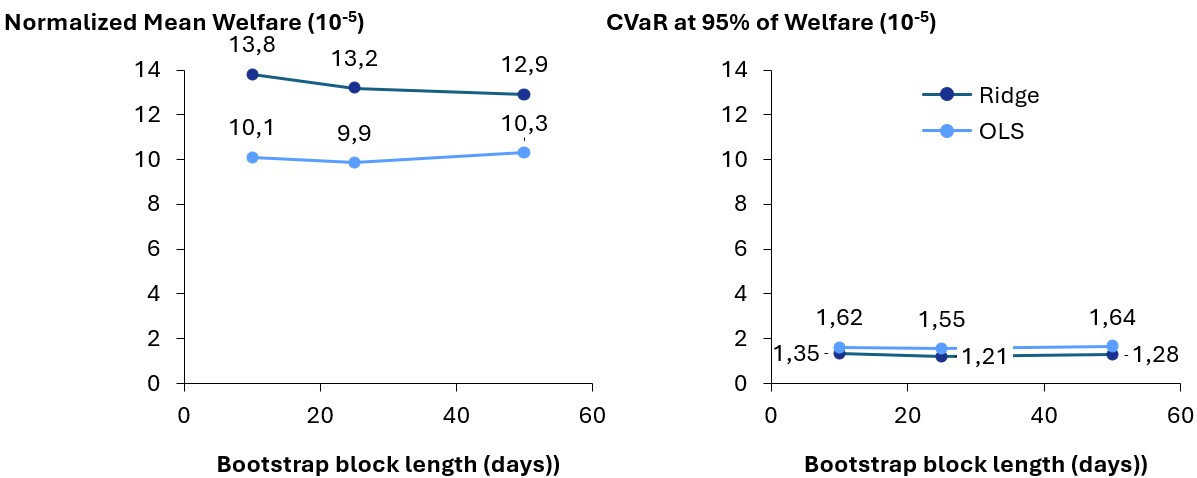}
\end{minipage}\hfill
\begin{minipage}[b]{0.36\linewidth}
\centering
\includegraphics[width=\linewidth]{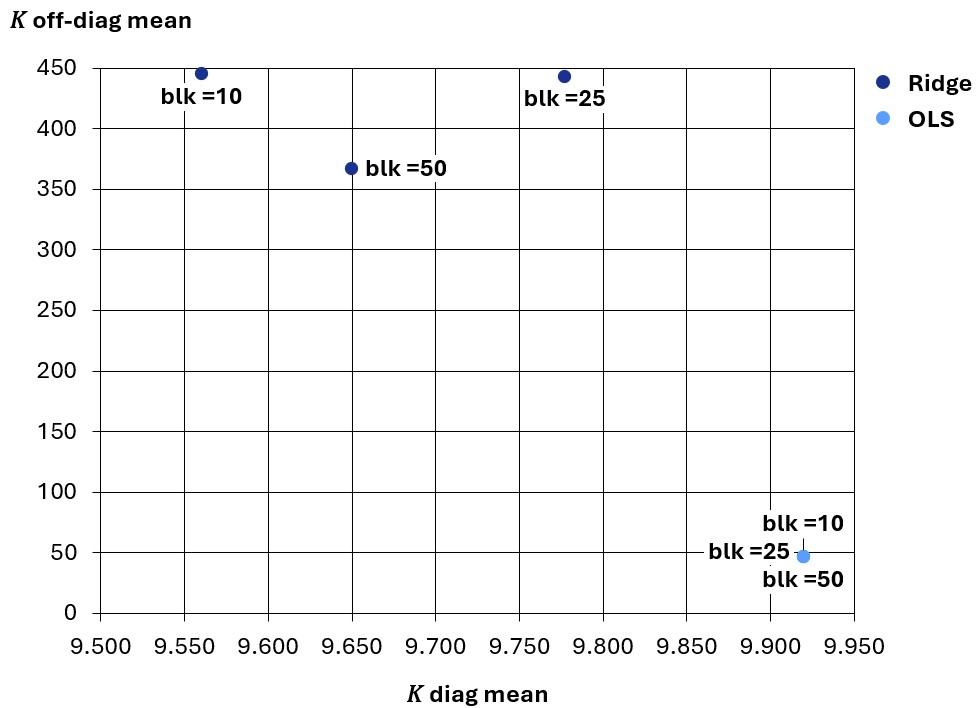}
\end{minipage}
\caption{Estimation robustness of the CVaR design across estimators (ridge vs.\ OLS) and bootstrap block lengths ($10$, $25$, $50$ days). \emph{Left (two panels):} Normalized mean welfare and CVaR@95\% of welfare versus block length, both in units of $10^{-5}$. Mean welfare is higher under ridge ($12.9$--$13.8$) than OLS ($9.9$--$10.3$), a gap of roughly $25$--$37\%$; the CVaR@95\% tail metric is positive under both estimators and comparable in magnitude ($1.2$--$1.4$ for ridge, $1.6$ for OLS). All four curves are nearly flat in block length, indicating that the results are stable to the assumed bootstrap dependence. \emph{Right:} Mean diagonal versus mean off-diagonal entries of the equilibrium map $K$. The diagonal (own-impact) means are similar across estimators ($\approx 9.5$--$9.9$), whereas the off-diagonal (cross-sector) means differ sharply: ridge retains richer cross-sector structure (off-diag mean $367$--$445$) while OLS collapses the off-diagonals to near zero (off-diag mean $\approx 47$). The qualitative CVaR findings---positive mean and tail welfare, stable across settings---hold under both estimators.}
\label{fig:estimation_robustness}
\end{figure}
Ridge regression yields consistently higher mean welfare ($1.29$--$1.38\times10^{-4}$) than OLS ($0.99$--$1.03\times10^{-4}$), a gap of roughly $25$--$37\%$, and a correspondingly richer cross-sector equilibrium map (off-diagonal $K$ mean $367$--$445$ vs.\ $47$ for OLS). OLS converges to a near-diagonal $K$ at every block length, indicating that it does not reliably identify cross-sector interactions from the available price history; ridge shrinkage, by contrast, stabilizes the off-diagonal $H$ estimates and preserves these couplings. The two estimators trade off along different margins: ridge attains higher mean welfare, while OLS attains a marginally higher (better) tail metric (CVaR@95\% $\approx1.6\times10^{-5}$ vs.\ $1.2$--$1.4\times10^{-5}$ for ridge). Within each estimator class the bootstrap block length has only a minor effect---mean welfare varies by under $5\%$ and the tail metric by at most about $12\%$ across block lengths of 10, 25, and 50 trading days---confirming that the results are qualitatively stable to reasonable assumptions about temporal dependence.

\subsection{Discussion}
The experiments establish that risk-sensitive information design is computationally tractable and structurally effective. The baseline retains a negative welfare tail (CVaR@95\%$<0$); the probabilistic design removes it and raises mean welfare sharply, while the CVaR design attains a positive CVaR@95\% with $P(w>0)=99.3\%$. Both reshape the equilibrium map relative to the baseline, but in opposite ways: the probabilistic design compresses $K$ toward near-diagonal in this calibration, whereas the CVaR design retains a richer off-diagonal structure that it uses for direct tail protection.
 
\paragraph{Financial interpretation.}
The objects above have concrete financial readings. Welfare $F\bullet X$ is the expected \emph{aggregate} risk-adjusted payoff of the sector system---the sum across sectors of return-type benefits net of quadratic interaction costs---under a chosen information-disclosure policy; the designer is therefore a market-information architect (an exchange, index provider, or regulator deciding what to reveal about fundamentals) rather than a trader. The off-diagonal action covariances $X_{ij}$ that the decentralization theorem caps are cross-sector co-movements---the channels through which a shock in one sector propagates to others---so bounding them is \emph{contagion suppression}. The tail metric is the cleanest reading: a positive CVaR@95\% means the designed system produces positive aggregate welfare even in the worst $5\%$ of fundamental realizations, a systemic-stability statement rather than a return forecast. Two qualifications matter for a financial reader: welfare is utilitarian aggregate quadratic payoff, not portfolio P\&L, so a rise in mean welfare under a sector stress shock reflects redistribution across the sum rather than a profitable crash; and the ETF data calibrate the payoff interactions $H$, not a tradeable strategy.
 
\paragraph{Economic mechanism.}
The contrast between designs is consistent with the joint decentralization theorem (Theorem~\ref{thm:decentralization}): both cap $|X_{ij}|\le\varepsilon/(\kappa_\beta\sigma_{ij})$, and although CVaR caps more tightly at a \emph{matched} tolerance (part (b)), the operative $\varepsilon$ is each design's feasibility threshold $\varepsilon^\star$ (part (c)), so the realized ordering is calibration-dependent. Where the goal is tail protection rather than decoupling per se, the CVaR design is the right tool: it keeps the system net-positive under adverse fundamentals and contains stress amplification. In this sense risk-aware information design acts as a macroprudential stabilizer---it reduces cross-sector interdependence rather than activity.
 
\paragraph{Which design protects welfare.}
The two formulations protect welfare through opposite mechanisms, so the recommendation depends on the welfare metric. The CVaR design shapes the tail of a large, active system: it keeps a high action scale ($\mathrm{tr}(K\Sigma K^\top)\approx3.0\times10^5$) and turns the baseline's negative tail ($-3.83$) positive ($+3.37$) while retaining $89\%$ of baseline mean welfare---large downside protection for an $11\%$ give-up in mean. The probabilistic design protects welfare instead by \emph{shrinking the action} (a $\sim100\times$ smaller scale, $\mathrm{tr}\approx2.9\times10^3$): it maximizes welfare per unit of action variance---a reward-to-activity ratio---and guarantees $P(w>0)=1$, but yields smaller absolute welfare. This is why the dominance ordering flips: CVaR dominates on absolute welfare and absolute tail, the probabilistic design on variance-normalized welfare, because normalization rewards its small footprint. A designer who wants the system to \emph{produce} aggregate welfare while bounding systemic losses should use the CVaR design at the highest credible confidence ($\beta=0.99$, where the $\beta$-sweep buys deeper tail protection at under $3\%$ mean cost); one who values risk-normalized efficiency and a minimal market footprint should prefer the probabilistic design. Two caveats: CVaR protects the tail only when welfare is the objective---under contagion suppression its welfare tail is the worst in the study (CVaR@95\%$=-15.0$), worse than the baseline---and the raw-versus-normalized flip rests on the action-scale gap, so the normalization must be chosen deliberately. As with the decentralization ordering, the better design depends on the regime, not on one formulation dominating outright.
 
\paragraph{Comparison to robust optimization.}
The distributional CVaR formulation contrasts with a prior worst-case robust information design \cite{sezer_robust2023}, which optimizes welfare under the most adverse $H$ in an uncertainty set: that approach gives absolute guarantees but can be overly conservative, solving a $\min_{H\in\Xi}$ problem, whereas the CVaR model enforces feasibility only over the worst $(1-\beta)$ tail of $H\sim\xi$. The CVaR design thus sits between full robustness and pure efficiency, with $\beta$ tuning the protection--performance trade-off, and with the corrector distribution $\lambda(\alpha|\theta)$ of Cal-BCE realigning recommendations with agents' true incentives when the designer's knowledge of $H$ is imperfect---keeping the design equilibrium-consistent under model misspecification.
 
Taken together, the sensitivity (Section~\ref{sec:beta_sweep}), alternative-objective (Section~\ref{sec:obj_sweep}), stress (Section~\ref{sec:stress}), and estimation-robustness (Section~\ref{sec:estimation_robustness}) experiments confirm three findings: (i) the probabilistic formulation maximizes expected welfare but at higher variance; (ii) the CVaR formulation trades mean efficiency for tail protection, with positive CVaR@95\% and contained shock amplification; and (iii) ridge estimation yields richer equilibrium structure and higher welfare than OLS, stably across bootstrap block lengths. All programs ($n=15$, $J=1500$) solved within minutes in MOSEK, scaling near-linearly in $J$ and polynomially in $n$ (See Appendix \ref{app_sca}). More broadly, by shaping the flow of information one can reduce systemic fragility without suppressing productive interaction---positioning information design as a possible foundation for macroprudential disclosure policy.

\section{Conclusion}
This paper studied information design under structural payoff uncertainty in linear-quadratic-Gaussian games, with four contributions. First, I introduced the Calibrated Bayes Correlated Equilibrium (Cal-BCE) and, adapting the revelation principle for Bayes correlated equilibrium, showed that optimizing over information structures and corrector policies reduces without loss to optimizing over action distributions subject to Cal-BCE constraints, even when the designer does not know the agents' payoff coefficients. Second, I derived tractable convex SDP and SOCP reformulations under two-sided probabilistic and CVaR constraints, with feasibility guaranteed by a checkable Hadamard invertibility condition and an explicit tolerance threshold. Third, I proved a joint decentralization theorem: both formulations cap the off-diagonal action covariances $|X_{ij}|\le\varepsilon/(\kappa_\beta\sigma_{ij})$, so risk-aware design suppresses strategic interdependence without reducing activity; CVaR caps more tightly at a common tolerance, but the realized ordering is set by each design's feasibility threshold $\varepsilon^\star$ and is calibration-dependent. Fourth, experiments with fifteen sector ETFs (2015-2024) confirm these findings: both designs deliver positive welfare under normal and stress scenarios; the probabilistic design attains higher mean welfare while the CVaR design improves its tail metric monotonically in $\beta$ at under $3\%$ mean-efficiency cost; a $-5\sigma$ technology-sector stress test shows the CVaR design contains shock amplification; and the better design depends on the objective and welfare metric rather than dominating outright. Estimation-robustness analysis shows these differences are stable across ridge and OLS estimation and bootstrap block lengths, with ridge yielding richer structure and higher welfare. Financially, welfare here is the aggregate risk-adjusted payoff of the sector system under a chosen information-disclosure policy, the capped cross-agent covariances are contagion channels, and a positive tail metric means the system stays net-positive under adverse fundamentals; the designer is thus a market-information architect rather than a trader. By shaping the flow of information, such a designer can reduce cross-sector interdependence rather than activity---suggesting a role for risk-aware information design in macroprudential disclosure and the mitigation of systemic fragility. Future work could extend the framework to dynamic (continuous-time) environments, alternative risk measures, sender credibility and learning, and richer datasets with macro and factor signals.

\bibliographystyle{plain}
\bibliography{ref}

\appendix
\smallskip
\section{Notation}\label{app_notation}
Here is a table introducing the key notation of the paper.
\begin{table}[h]
\small
\centering
\caption{Summary of key notation.}
\label{tab:notation}
\begin{tabular}{ll}
\toprule
Symbol & Description \\
\midrule
$\zeta(\omega|\gamma)$ & Information structure: distribution of signals $\omega$ given payoff state $\gamma$ \\
$\phi(a|\gamma)$ & Action distribution: probability of action profile $a$ given payoff state $\gamma$ \\
$\lambda(\alpha|\theta)$ & Corrector distribution: adjustment to actions accounting for unknown payoff coefficients $\theta$ \\
$X \in \mathbb{S}_+^{2n}$ & Covariance block for $(a,\gamma)$: variances and covariances of actions and payoff states \\
$G \in \mathbb{S}_+^{2n}$ & Covariance block for $(\alpha,\theta)$: variances and covariances of corrector terms and payoff coefficients \\
$H \in \mathbb{R}^{n\times n}$ & Payoff coefficient matrix: self-costs (diagonal) and cross-impacts (off-diagonal) \\
$\theta=\mathrm{diag}(H)\in\mathbb{R}^n$ & Uncertainty state: vector of own-coefficients $\theta_i=H_{ii}$, known to agents, unknown to designer \\
$\alpha\in\mathbb{R}^n$ & Corrector profile: per-agent additive correction to recommended actions, drawn from $\lambda(\alpha|\theta)$ \\
$K \in \mathbb{R}^{n\times n}$ & Equilibrium map: sensitivity of equilibrium actions to payoff states \\
$t \in \mathbb{R}$ & Auxiliary variable for system objective lower bound \\
$\beta_F,\beta_i \in (0,1)$ & Confidence levels for system objective and equilibrium constraints \\
$\varepsilon,\varepsilon_X,\varepsilon_G \ge 0$ & Tolerances (band half-widths) for the two-sided equilibrium constraints \\
\bottomrule
\end{tabular}
\end{table}

\section{Existence of Nontrivial Calibrated Bayes Correlated Equilibrium (Cal-BCE)}
\label{app:Cal-BCE_existence}
To establish that Cal-BCE exists in nontrivial form when the designer faces parameter uncertainty, a constructive existence result based on diagonal linear policies with explicit correction structure. The construction shows how the corrector distribution $\lambda(\alpha|\theta)$ can be chosen to offset uncertainty in payoff coefficients and maintain equilibrium consistency, and it specializes to a closed form with gains $1/\bar H_{ii}$ when the covariances are diagonal (part (d)).

\begin{theorem}[Existence of a nontrivial Cal-BCE under general covariances]\label{thm:general_existence}
Let $\bar H=\mathbb{E}_\xi[H]$, $\Sigma=\mathrm{var}(\gamma)$, $\Psi=\mathrm{var}(\theta)$, write $A\circ B$ for the Hadamard (entrywise) product, and let $\mathrm{dg}(M)=(M_{11},\dots,M_{nn})^\top$ collect the diagonal entries of $M$. Suppose $\bar H\circ\Sigma$ and $\bar H\circ\Psi$ are invertible, and set
\[
d=(\bar H\circ\Sigma)^{-1}\mathrm{dg}(\Sigma),\qquad
b=(\bar H\circ\Psi)^{-1}\mathrm{dg}(\Psi),
\]
with $D=\mathrm{Diag}(d)$, $B=\mathrm{Diag}(b)$, and the diagonal policies $a=D\gamma$, $\alpha=B(\theta-\mathbb{E}[\theta])$. Then:
\begin{enumerate}
\item[(a)] \emph{(Feasible covariance.)} $X:=\begin{bmatrix}D\\ I\end{bmatrix}\Sigma\begin{bmatrix}D\\ I\end{bmatrix}^{\top}\succeq0$ has $X_{\gamma\gamma}=\Sigma$ and satisfies the covariance-consistency constraint \eqref{eq_cov_gamma}; the corrector block is built analogously from $(B,\Psi)$.
\item[(b)] \emph{(Equilibrium in mean; exact under certainty.)} $\Lambda^{(X)}_i(H)=d_i\big[((H\circ\Sigma)d)_i-\Sigma_{ii}\big]$, so $\mathbb{E}_\xi[\Lambda^{(X)}_i(H)]=0$ for every $i$, and likewise $\mathbb{E}_\xi[\Lambda^{(G)}_i(H)]=0$. If $H$ is known ($\Sigma_H=0$) the residuals vanish pathwise and $X$ is an exact, nontrivial Cal-BCE ($a\not\equiv0$, $\alpha\not\equiv0$).
\item[(c)] \emph{(Robust feasibility above a threshold.)} Under Gaussian $H$, $\Lambda^{(X)}_i(H)$ is mean-zero Gaussian with standard deviation $s_i=d_i\big(\sum_j \sigma_{ij}^2\Sigma_{ij}^2 d_j^2\big)^{1/2}$. The two-sided chance program (Theorem~\ref{theorem_chance}) and the CVaR program (Theorem~\ref{theorem_cvar}) are feasible for every tolerance
\[
\varepsilon\ \ge\ \varepsilon^\star:=\kappa_\beta\,\max_i s_i ,
\]
with $\kappa_\beta$ the multiplier of Theorem~\ref{thm:decentralization}; the corrector channel admits the analogous threshold with $\Psi$ in place of $\Sigma$. Hence the robust feasible set is nonempty above an explicit, coefficient-uncertainty-determined tolerance.
\item[(d)] \emph{(Diagonal case: closed form and zero-tolerance exactness.)} If $\Sigma$ and $\Psi$ are diagonal then $\bar H\circ\Sigma=\mathrm{Diag}(\bar H_{ii}\Sigma_{ii})$, so the gains reduce to $d_i=b_i=1/\bar H_{ii}$ and the off-diagonal action covariances $X_{ij}$ ($i\neq j$) vanish. The residual collapses to the own-coefficient term $H_{ii}X_{ii}-X_{i,n+i}$, which the corrector neutralizes exactly ($b_i=1/H_{ii}$). In this case the Cal-BCE moment identities hold \emph{pathwise} in $H$---not merely in mean---so the chance and CVaR constraints are met with zero shortfall at tolerance $\varepsilon=0$, giving a nontrivial Cal-BCE ($a\not\equiv0$, $\alpha\not\equiv0$) with independent (componentwise) shocks.
\end{enumerate}
\end{theorem}

\begin{proof}{Proof}
(a) $X=M\Sigma M^\top$ with $M=\begin{bmatrix}D\\ I\end{bmatrix}$ is positive semidefinite because $\Sigma\succeq0$; its lower-right block is $I\Sigma I=\Sigma$, so the consistency constraint \eqref{eq_cov_gamma} holds. The corrector block is identical with $(B,\Psi)$ replacing $(D,\Sigma)$.

(b) From $X_{ij}=\mathrm{cov}(a_i,a_j)=d_id_j\Sigma_{ij}$ and $X_{i,n+i}=\mathrm{cov}(a_i,\gamma_i)=d_i\Sigma_{ii}$,
\[
\Lambda^{(X)}_i(H)=\sum_j H_{ij}\,d_id_j\Sigma_{ij}-d_i\Sigma_{ii}
=d_i\Big[\big((H\circ\Sigma)d\big)_i-\Sigma_{ii}\Big].
\]
Taking expectations and using $(\bar H\circ\Sigma)d=\mathrm{dg}(\Sigma)$ gives $\mathbb{E}_\xi[\Lambda^{(X)}_i]=0$. If $\Sigma_H=0$ then $H=\bar H$ almost surely and the bracket vanishes pathwise. The corrector identity follows verbatim with $(\alpha,\theta,B,\Psi)$ in place of $(a,\gamma,D,\Sigma)$.

(c) Only the term $d_i\sum_j \Sigma_{ij}d_j H_{ij}$ is random; with independent $H_{ij}$ of variance $\sigma_{ij}^2$, $\mathrm{Var}[\Lambda^{(X)}_i]=d_i^2\sum_j \sigma_{ij}^2\Sigma_{ij}^2 d_j^2=s_i^2$. By the two-sided chance/CVaR reformulations, constraint $i$ holds once $\kappa_\beta s_i\le\varepsilon-|\mathbb{E}[\Lambda^{(X)}_i]|=\varepsilon$; thus $\varepsilon\ge\kappa_\beta\max_i s_i$ suffices, and the objective is then maximized over a nonempty feasible region.

(d) For diagonal $\Sigma$, $(\bar H\circ\Sigma)$ is diagonal with entries $\bar H_{ii}\Sigma_{ii}$, whence $d_i=\Sigma_{ii}/(\bar H_{ii}\Sigma_{ii})=1/\bar H_{ii}$ and $X_{ij}=d_id_j\Sigma_{ij}=0$ for $i\neq j$. The residual from (b) is then $\Lambda^{(X)}_i(H)=d_i[H_{ii}d_i\Sigma_{ii}-\Sigma_{ii}]=d_i\Sigma_{ii}(H_{ii}d_i-1)$, which vanishes pathwise once $d_i=1/H_{ii}$; the corrector identity vanishes identically for $b_i=1/H_{ii}$ by the same computation with $(\alpha,\theta,b,\Psi)$ in place of $(a,\gamma,d,\Sigma)$. Hence both moment equalities hold for every realization of $H$, and the chance/CVaR constraints are satisfied with zero shortfall at $\varepsilon=0$. \qed
\end{proof}

\begin{remark}[Interpretation and scope]
Theorem~\ref{thm:general_existence} provides a constructive, nontrivial existence result with meaningful correction. The diagonal gains $d=(\bar H\circ\Sigma)^{-1}\mathrm{dg}(\Sigma)$ and $b=(\bar H\circ\Psi)^{-1}\mathrm{dg}(\Psi)$ enforce the Cal-BCE identities in mean whenever the Hadamard products $\bar H\circ\Sigma$ and $\bar H\circ\Psi$ are invertible, and the robust programs are then feasible for every tolerance $\varepsilon\ge\kappa_\beta\max_i s_i$. In the componentwise case (diagonal $\Sigma$ and $\mathrm{var}(\theta)$), part (d) sharpens this to a closed form: a diagonal linear policy ($a_i\propto\gamma_i$) with a nonzero, agentwise correction ($\alpha_i\propto\theta_i$) and gains $1/\bar H_{ii}$ enforces the identities \emph{pathwise} in $H$, at zero tolerance. The Hadamard condition is sufficient, not necessary: when it fails, a non-diagonal or sparse gain may still certify feasibility.
\end{remark}

\section{First-moment (mean) conditions}\label{app_mean}
\begin{remark}[First-moment (mean) conditions]\label{rem:first-moment}
Lemmas~\ref{lem:linear_cvar_X}--\ref{lem_prob} project the agents' obedience first-order condition onto the centered instruments $(a_i-\mathbb{E}[a_i])$ and $(\alpha_i-\mathbb{E}[\alpha_i])$, which yields the second-moment constraints stated in those lemmas. Projecting the same first-order condition onto the constant instrument (i.e.\ taking unconditional expectations, using $\mathbb{E}[m_{ij}]=\mu_{a,j}$ and $\mathbb{E}[m_{i\gamma}]=\mu_{\gamma,i}$) yields the complementary \emph{first-moment} balance
\begin{equation}\label{eq:mean_balance}
\Lambda^{(\mu)}_i(H):=\sum_{j\in\mathcal N} H_{ij}\,\mu_{a,j}-H_{ii}\,\mu_{\alpha,i}-\mu_{\gamma,i}=0,\qquad i\in\mathcal N,
\end{equation}
where $\mu_a=\mathbb{E}[a]$, $\mu_\alpha=\mathbb{E}[\alpha]$, and $\mu_\gamma=\mathbb{E}[\gamma]=\mu$. Since $\mathbb{E}[a_j]$ and $\mathbb{E}[\gamma_i]$ are nonzero in general, this condition is not implied by the second-moment constraints of Lemmas~\ref{lem:linear_cvar_X}--\ref{lem_prob} and must be imposed separately. Two regimes arise.

\emph{(Centered design; maintained in this paper.)} Take $\mu=0$ and a mean-zero corrector $\mathbb{E}[\alpha]=0$. The latter is without loss for the second-moment conditions: replacing $\alpha_i$ by its centered version $\alpha_i-\mathbb{E}[\alpha_i]$ (equivalently, calibrating to the fluctuation $\theta_i-\mathbb{E}[\theta_i]$ rather than the level of the uncertain own-coefficient) leaves $\mathrm{cov}(\alpha_i,\alpha_j)$ and $\mathrm{cov}(\alpha_i,\theta_i)$, and hence the corrector-channel constraint \eqref{eq_affected_inequality_app_cvar_2}, unchanged. Then \eqref{eq:mean_balance} reduces to $\sum_j H_{ij}\mu_{a,j}=0$, which is satisfied by $\mu_a=0$. Moreover, when $\mu=0$ the objective \eqref{eq_obj_exp} has no first-order term in $\mu_a$ (its mean part is the quadratic $\mu_a^\top[F]_{1,1}\mu_a$), so $\mu_a=0$ is a stationary point and is optimal for the concave (welfare-based) objective of Examples~\ref{ex_social_welfare}, while the contagion suppression objective of Example~\ref{ex_contagion} is mean-invariant. Hence the first-moment balance holds at the optimum and the design reduces \emph{without loss} to the covariance programs of Theorems~\ref{theorem_cvar}--\ref{theorem_chance}, with $\mathbb{E}_\phi[f]=F\bullet X$ exact.

\emph{(General means.)} For $\mu\neq0$, the first-moment balance \eqref{eq:mean_balance} is affine in $H$ and in the mean variables $(\mu_a,\mu_\alpha)$, so it can be appended to the CVaR and chance programs by the same Rockafellar--Uryasev/SAA and second-order-cone reformulations used for the second-moment constraints, with the objective acquiring the term $z^\top F z$, $z=[\mu_a;\mu]$. Linearity in a single decision matrix is preserved by lifting $X$ to the augmented moment matrix $\bar X=\mathbb{E}[[1;a;\gamma][1;a;\gamma]^\top]\in\mathbb{S}_+^{2n+1}$ with $\bar X_{11}=1$, whose border carries the means and whose interior block carries the second moments; the centered case is recovered by fixing the border to zero.
\end{remark}
\newpage
\section{ETF Summary}
\label{app_etf}
Basic statistics regarding sector ETF data we used in numerical experiments are provided.
\begin{table}[h]
\centering
\caption{Sample length (number of trading days), annualized mean log return (in percent per annum), and annualized volatility (percent per annum) for the selected ETFs. Statistics are computed from daily log returns using Yahoo Finance price series.}
\footnotesize
\label{tab:etf_stats}
\begin{tabular}{l l r r r}
\toprule
Ticker & Sector Name & Sample length (days) & Mean (\% p.a.) & Volatility (\% p.a.)\\
\midrule
XLK  & Technology & 2697 & 19.047 & 23.807\\
XLE  & Energy & 2697 & 5.185 & 29.789\\
XLF  & Financials & 2697 & 11.058 & 22.152\\
XAR  & Aerospace \& Defense & 2697 & 14.322 & 23.393\\
XBI  & Biotech & 2697 & 4.269 & 33.228\\
XES  & Oil \& Gas & 2697 & -11.309 & 46.182\\
XLB  & Materials & 2697 & 7.661 & 20.712\\
FCOM & Communications & 2697 & 11.147 & 20.627\\
XLI  & Industrials & 2697 & 11.082 & 19.742\\
XLP  & Consumer Staples & 2697 & 7.194 & 14.667\\
VNQ  & Real Estate & 2697 & 4.885 & 20.856\\
XLU  & Utilities & 2697 & 8.836 & 19.138\\
XLV  & Health Care & 2697 & 8.065 & 16.872\\
XLY  & Consumer Discretionary & 2697 & 12.404 & 21.644\\
XTN  & Transportation & 2697 & 5.145 & 25.505\\
\bottomrule
\end{tabular}
\end{table}

\begin{table}[ht]
\centering
\caption{Correlation matrix of daily log returns for the selected ETFs. Pairwise correlations are computed on the aligned overlapping sample of data from Jan 1, 2015 to Dec 30, 2024.}
\footnotesize
\label{tab:etf_corr}
\begin{tabular}{l r r r r r r r r r r r r r r r}
\toprule
 & XLK & XLE & XLF & XAR & XBI & XES & XLB & FCOM & XLI & XLP & VNQ & XLU & XLV & XLY & XTN \\
\midrule
XLK & 1.00 & 0.47 & 0.68 & 0.67 & 0.59 & 0.42 & 0.70 & 0.79 & 0.75 & 0.56 & 0.59 & 0.43 & 0.67 & 0.84 & 0.67 \\
XLE & 0.47 & 1.00 & 0.69 & 0.64 & 0.37 & 0.88 & 0.68 & 0.48 & 0.69 & 0.42 & 0.49 & 0.36 & 0.48 & 0.50 & 0.60 \\
XLF & 0.68 & 0.69 & 1.00 & 0.78 & 0.50 & 0.61 & 0.82 & 0.67 & 0.88 & 0.62 & 0.67 & 0.49 & 0.69 & 0.73 & 0.78 \\
XAR & 0.67 & 0.64 & 0.78 & 1.00 & 0.55 & 0.59 & 0.75 & 0.65 & 0.88 & 0.54 & 0.66 & 0.48 & 0.61 & 0.72 & 0.76 \\
XBI & 0.59 & 0.37 & 0.50 & 0.55 & 1.00 & 0.37 & 0.52 & 0.55 & 0.53 & 0.35 & 0.46 & 0.26 & 0.65 & 0.61 & 0.55 \\
XES & 0.42 & 0.88 & 0.61 & 0.59 & 0.37 & 1.00 & 0.62 & 0.45 & 0.62 & 0.30 & 0.39 & 0.23 & 0.36 & 0.45 & 0.58 \\
XLB & 0.70 & 0.68 & 0.82 & 0.75 & 0.52 & 0.62 & 1.00 & 0.67 & 0.88 & 0.64 & 0.68 & 0.52 & 0.69 & 0.74 & 0.78 \\
FCOM & 0.79 & 0.48 & 0.67 & 0.65 & 0.55 & 0.45 & 0.67 & 1.00 & 0.70 & 0.56 & 0.62 & 0.45 & 0.61 & 0.80 & 0.67 \\
XLI & 0.75 & 0.69 & 0.88 & 0.88 & 0.53 & 0.62 & 0.88 & 0.70 & 1.00 & 0.66 & 0.71 & 0.55 & 0.72 & 0.78 & 0.85 \\
XLP & 0.56 & 0.42 & 0.62 & 0.54 & 0.35 & 0.30 & 0.64 & 0.56 & 0.66 & 1.00 & 0.68 & 0.70 & 0.71 & 0.59 & 0.51 \\
VNQ & 0.59 & 0.49 & 0.67 & 0.66 & 0.46 & 0.39 & 0.68 & 0.62 & 0.71 & 0.68 & 1.00 & 0.73 & 0.64 & 0.66 & 0.63 \\
XLU & 0.43 & 0.36 & 0.49 & 0.48 & 0.26 & 0.23 & 0.52 & 0.45 & 0.55 & 0.70 & 0.73 & 1.00 & 0.55 & 0.43 & 0.38 \\
XLV & 0.67 & 0.48 & 0.69 & 0.61 & 0.65 & 0.36 & 0.69 & 0.61 & 0.72 & 0.71 & 0.64 & 0.55 & 1.00 & 0.65 & 0.58 \\
XLY & 0.84 & 0.50 & 0.73 & 0.72 & 0.61 & 0.45 & 0.74 & 0.80 & 0.78 & 0.59 & 0.66 & 0.43 & 0.65 & 1.00 & 0.77 \\
XTN & 0.67 & 0.60 & 0.78 & 0.76 & 0.55 & 0.58 & 0.78 & 0.67 & 0.85 & 0.51 & 0.63 & 0.38 & 0.58 & 0.77 & 1.00 \\
\bottomrule
\end{tabular}
\end{table}
\newpage
\section{Scalability Study}
\label{app_sca}
Here, the scaling of runtime as we increase the number of payoff scenarios and the number of sectors is shown.
\begin{figure}[ht]
\centering
\includegraphics[width=0.8\linewidth]{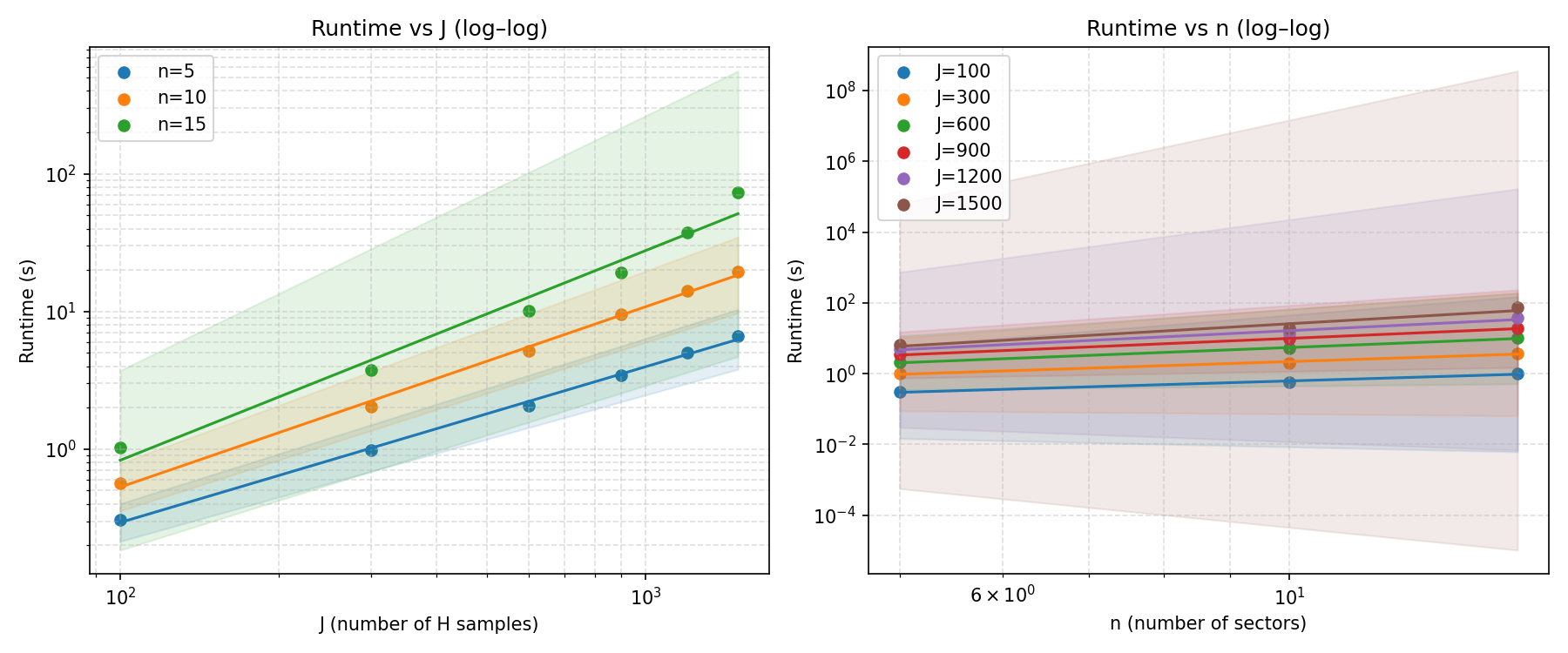}
\caption{Runtime scaling of the convex information design formulations (log--log scale). 
\emph{Left:} Runtime as a function of the number of sampled payoff scenarios $J$, for fixed numbers of sectors $n\in\{5,10,15\}$. 
Scaling exponents (slope of fitted line) range from $1.15$ (blue, $n=5$) to $1.45$ (green, $n=15$), consistent with nearly linear growth in $J$. 
\emph{Right:} Runtime as a function of the number of sectors $n$, for fixed scenario counts $J\in\{100,300,600,900,1200,1500\}$. 
Scaling exponents range between $1.1$ and $1.8$, consistent with polynomial growth in $n$. 
Shaded bands show 95\% confidence intervals from regression fits. 
Overall, the empirical scaling confirms that solve times increase linearly in the number of samples and polynomially in the number of sectors, remaining tractable (on the order of seconds to minutes) even for $n=15$ sectors and $J=1500$ scenarios.}
\label{fig:runtime_scaling}
\end{figure}








\end{document}